\documentclass[11pt]{amsart} \usepackage{amscd,amssymb}   

\usepackage{color}

\definecolor{darkgreen}{cmyk}{1,0,1,.2}
\definecolor{m}{rgb}{1,0.1,1}

\usepackage[matrix, arrow]{xy}  
\usepackage{array} 
\usepackage{pb-diagram} 
 
\begin{document} 

\def\dach{\!\widehat{ \;\;\; }}
\def\A{\mathcal{A}} 
\def\G{\mathcal{G}} 
\def\C{\mathcal{C}} 
\def\H{\mathcal {H}}
\def\d{\operatorname{d}}
\def\im{\operatorname{im}}
\def\iDelta{\stackrel{\circ}{\Delta}}
\def\bDelta{\partial\Delta}
\def\const{\operatorname{const}}

\def\m{{\boldmath m}}
\def\l{{\boldmath l}}
\def\h{{\boldmath h}}
\def\x{{\boldmath x}}
\def\y{{\boldmath y}}
\def\ab{\operatorname{ab}}

\def\tg{\operatorname{tg}}

\def\Call{\mathcal C_{\operatorname{all}}}
\def\Csep{\mathcal C_{\operatorname{sep}}}
\def\Cnuc{\mathcal C_{\operatorname{nuc}}}
\def\Ccom{\mathcal C_{\operatorname{com}}}
\def\Ob{\operatorname{Ob}}
\def\Mor{\operatorname{Mor}} 
\def\Aut{\operatorname{Aut}} 
\def\cell{\operatorname{cell}} 
\def\Cell{\operatorname{Cell}} 
\def\incl{\operatorname{incl}} 
\def\sign{\operatorname{sign}} 
\def\Gx{X\rtimes G}  
\def\Bott{\operatorname{Bott}} \def\CC{\mathbb C} 
\def\ClV{C_{\!{}_{V}}} \def\comp{\operatorname{comp}} 
\def\ddS{\stackrel{\scriptscriptstyle{o}}{S}} 
\def\intW{\stackrel{\scriptscriptstyle{o}}{W}} 
\def\intU{\stackrel{\scriptscriptstyle{o}}{U}} 
\def\intT{\stackrel{\scriptscriptstyle{o}}{T}} \def\E{\mathcal E} 
\def\EE{\mathbb E} \def\EG{\mathcal{E}(G)} 
\def\Egx{\mathcal{E}(G)\times X} 
\def\EGG{\mathcal{E}(\mathcal{G})} \def\EGT{\mathcal{E}(G/G_0)}\def\Eg{\mathcal{E}(G)} 
\def\EGN{\mathcal{E}(G/N)} \def\EH{\mathcal{E}(H)} \def\F{\mathcal F} 
\def\I{{\operatorname{IND}}} \def\Ind{\operatorname{Ind}} 
\def\Id{\operatorname{Id}} 
\def\infl{\operatorname{inf}} \def\k{\operatorname{K}} 
\def\KK{\operatorname{KK}} \def\L{\mathcal L} \def\lk{\langle} 
\def\rk{\rangle} \def\NN{\mathbb N}
\renewcommand{\oplus}{\bigoplus} 
\def\pt{\operatorname{pt}} \def\QQ{\mathbb Q} 
\def\res{\operatorname{res}} \def\RKK{\mathcal{R}\!\operatorname{KK}} 
\def\RR{\mathbb R} \def\sm{\backslash} \def\top{\operatorname{top}} 
\def\ZM{{\mathcal Z}M} \def\ZZ{\mathbb Z} \setcounter{section}{-1} 
\def\Inf{\operatorname{Inf}} \def\Ad{\operatorname{Ad}} 
\def\K{\mathcal K} \def\id{\operatorname{id}} 
\def\U{\mathcal U}
\def\PU{\mathcal PU}
\def\BG{\operatorname{BG}}
\def\eps{\epsilon} 
\def\om{\omega}
\def\F{\mathcal{F}} 
\def\TT{\mathbb T}
\def\Flk{\F_{L,K}} 
\def\mg{\mu_{G,A}} 
\def\mgx{\mu_{\Gx,A}} 
\def\ts{\otimes} 
\def\ga{\gamma} 
\def\Br{\operatorname{Br}}
\def\Ab{\operatorname{Ab}}
\def\ev{\operatorname{ev}}
\def\la{\lambda} 
\def\oDelta{\stackrel{\circ}{\Delta}}
\emergencystretch= 30 pt
\theoremstyle{plain} \newtheorem{thm}{Theorem}[section] 
\newtheorem{cor}[thm]{Corollary} \newtheorem{lem}[thm]{Lemma} 
\newtheorem{prop}[thm]{Proposition} \newtheorem{lemdef}[thm]{Lemma and 
  Definition} \theoremstyle{definition} 
\newtheorem{defn}[thm]{Definition}  
\newtheorem{defremark}[thm]{\bf Definition and Remark} 
\theoremstyle{remark} 
\newtheorem{remark}[thm]{\bf Remark} \newtheorem{ex}[thm]{\bf Example} 
\newtheorem{notations}[thm]{\bf Notations} 
\numberwithin{equation}{section} \emergencystretch 25pt 
\renewcommand{\theenumi}{\roman{enumi}} 
\renewcommand{\labelenumi}{(\theenumi)} \title[On noncommutative torus
bundles]{Principal noncommutative torus
bundles}   
\author[Echterhoff]{Siegfried 
  Echterhoff} 
\address{Westf\"alische Wilhelms-Universit\"at M\"unster, 
  Mathematisches Institut, Einsteinstr. 62 D-48149 M\"unster, Germany} 
\email{echters@math.uni-muenster.de}
\author[Nest]{Ryszard Nest} \address{SNF Center in
Non-Commutative Geometry
Universitetsparken 5
DK-2100 KBH. \O, Denmark}\email{rnest@math.ku.dk}
\author[Oyono-Oyono]{Herve Oyono-Oyono}
\address{Universit\'e Blaise Pascal de  Clermont-Ferrand, Laboratoire
  de Math\'ematiques, Plateau des C\'ezeaux, 63177 Aubi\`ere Cedex,
  France}
\email{oyono@math.cnrs.fr}

\thanks{This work was partially supported by the Deutsche Forschungsgemeinschaft
(SFB 478)}
 



\begin{abstract}
In this paper we study continuous bundles of C*-algebras which are 
non-commutative analogues of principal torus bundles. We show that 
all such bundles, although in general being very far away from being 
locally trivial bundles, are at least locally $\RKK$-trivial.
Using earlier results of Echterhoff and Williams, we shall give 
a complete classification of principal non-commutative torus bundles 
up to $\TT^n$-equivariant Morita equivalence. We then study 
these bundles as topological fibrations (forgetting the group action)
and give necessary and sufficient conditions for any non-commutative 
principal torus bundle being $\RKK$-equivalent to a commutative one.
As an application of our methods we shall also give a $K$-theoretic 
characterization of those principal $\TT^n$-bundles with $H$-flux,
as studied by Mathai and Rosenberg  which 
{possess} "classical" $T$-duals. 
\end{abstract} 
\maketitle  
 \section{Introduction}
 In this paper we want to start a general study of C*-algebra bundles over locally
 compact base spaces $X$ which are non-commutative analogues of Serre fibrations in 
 topology. Before we shall proof some general results for analysing such bundles in 
 a forthcoming paper \cite{ENO}  (e.g., by proving a general spectral sequence for computing
 the $K$-theory of the algebra of the total bundle) we want to introduce in this 
 paper our most basic (and probably also most interesting) toy examples, namely
the  non-commutative analogues of 
 principal torus bundles as defined below.
 
 Recall that a locally compact principal $\TT^n$-bundle $q:Y\to X$ 
 consists of a locally compact space $Y$ 
 equipped with a free action of $\TT^n$ on $Y$ such that 
 $q:Y\to X$ identifies the orbit space $\TT^n\backslash Y$ with $X$.
 In order to introduce a suitable non-commutative analogue, we 
 recall that from Green's theorem   \cite{Green0} we have 
 $$C_0(Y)\rtimes \TT^n\cong C_0(X,\K)$$
 for the crossed product by the action of $\TT^n$ on $Y$, where $\K$ denotes the algebra 
 of compact operators on the infinite dimensional separable Hilbert space. 
 Using this observation, we introduce NCP $\TT^n$-bundles
as follows

 \begin{defn}\label{defn-nctorus}
 By a (possibly) {\em noncommutative principal  $\TT^n$-bundle}
 (or {\em NCP $\TT^n$-bundle}) over $X$ we understand a separable 
C*-algebra bundle $A(X)$ 
 together with a fibre-wise  action $\alpha:\TT^n\to\Aut(A(X))$ such that 
 $$A(X)\rtimes_{\alpha}\TT^n\cong C_0(X,\K)$$
 as C*-algebra bundles over $X$.
 \end{defn}

We should note that by a C*-algebra bundle over $X$ we simply understand 
a $C_0(X)$-algebra, i.e., a C*-algebra $A$ which is equipped with a 
$C_0(X)$-bimodule structure given by a non-degenerate $*$-homomorphism
$\Phi_A:C_0(X)\to ZM(A)$, where $ZM(A)$ denotes the centre of the 
multiplier algebra $M(A)$ of $A$. The fibre $A_x$ of 
$A(X)$ over $x\in X$ is then defined as the quotient of $A(X)$
by the ideal $I_x:=\Phi\big(C_0(X\setminus\{x\})\big)A(X)$.
We call an action $\alpha:G\to \Aut(A(X))$   {\em fibre-wise}
if it is $C_0(X)$-linear in the sense that 
$$\alpha_s(\Phi_A(f)a)=\Phi_A(f)(\alpha_s(a))\quad\text{for all $f\in C_0(X)$, $a\in A$}.$$
This implies that $\alpha$ induces actions $\alpha^x$ on the fibres $A_x$ for all $x\in X$
and the crossed product $A(X)\rtimes_{\alpha}G$ is again a C*-algebra 
bundle (i.e. $C_0(X)$-algebra) with structure map $\Phi_{A\rtimes G}:C_0(X)\to ZM(A\rtimes G)$
given by the composition of $\Phi_A$ with the canonical inclusion $M(A)\subseteq M(A\rtimes G)$,
and with fibres $A_x\rtimes_{\alpha^x}G$ (at least if we consider the full crossed products).

If $A(X)$ is a NCP $\TT^n$-bundle,  then the crossed product
$A(X)\rtimes_{\alpha}\TT^n\cong C_0(X,\K)$ comes equipped with the dual action of 
$\ZZ^n\cong \widehat{\TT^n}$, and then the Takesaki-Takai duality theorem tells us
that 
$$A(X)\sim_M C_0(X,\K)\rtimes_{\widehat{\alpha}}\ZZ^n$$
as C*-algebra bundles over $X$, where $\sim_M$ denotes $X\times\TT^n$-equivariant 
Morita equivalence. The corresponding action of $\ZZ^n$ on $C_0(X,\K)$ is also fibre-wise.
Conversely, if we start with any fibre-wise action $\beta:\ZZ^n\to\Aut(C_0(X,\K))$,
the Takesaki-Takai duality theorem implies that 
$$A(X)= C_0(X,\K)\rtimes_\beta  \ZZ^n$$
is a NCP $\TT^n$-bundle.  Thus, up to 
a suitable notion of Morita equivalence, the NCP $\TT^n$-bundles
are precisely the crossed products $C_0(X,\K)\rtimes_{\beta}\ZZ^n$ 
for some fibre-wise action $\beta$ of $\ZZ^n$ on $C_0(X,\K)$ and equipped with 
the dual $\TT^n$-action.  

Using this translation, the results of \cite{EW1, EW2}
provide a complete classification of NCP $\TT^n$-bundles
up to $\TT^n$-equivariant Morita equivalence in terms of pairs
$([q:Y\to X], f)$, where $[q:Y\to X]$ denotes the isomorphism class 
of a commutative principal $\TT^n$-bundle $q:Y\to X$ over $X$ and 
$f:X\to \TT^{\frac{n(n-1)}{2}}$ is a certain ``classifying'' map of the bundle (see \S \ref{sec-torus} below for more details).

In this paper we are mainly interested in the topological structure of the underlying 
non-commutative fibration $A(X)$ after forgetting the $\TT^n$-action.
Note that these bundles are in general quite irregular. 
To see this let us briefly discuss the Heisenberg bundle, which is probably 
the most prominent example of a NCP torus bundle:
Let $C^*(H)$ denote the C*-group algebra of the discrete Heisenberg group
$$H=\left\{\left(\begin{smallmatrix} 1 & n & k\\ 0& 1 & m\\ 0&0&1\end{smallmatrix}\right): 
n,m,k\in \ZZ\right\}.$$
The centre of this algebra is equal to $C^*(Z)$ with 
$Z=\left\{\left(\begin{smallmatrix} 1 & 0 & k\\ 0& 1 & 0\\ 0&0&1\end{smallmatrix}\right): 
k\in \ZZ\right\}\cong \ZZ$
the centre of $H$. Hence we have $Z(C^*(H))\cong C^*(\ZZ)\cong C(\TT)$ 
which implements a canonical C*-algebra bundle structure on $C^*(H)$ 
with base  $\TT$. It is well documented in the literature (e.g. see \cite{AP}) that the fibre
of this bundle at $z\in \TT$ is isomorphic to the non-commutative 2-torus $A_{\theta}$
if $z=e^{2\pi i\theta}$. Moreover, there is an obvious fibre-wise $\TT^2$-action on $C^*(H)$,
namely the dual action of $\widehat{H/Z}\cong \TT^2$, and one can check that
$$C^*(H)\rtimes\TT^2\cong C(\TT,\K)$$
as bundles over $\TT$.  Since the non-commutative $2$-tori $A_{\theta}$ differ substantially 
for different values of $\theta$ (e.g. they are simple for irrational $\theta$ and non-simple 
for rational $\theta$), the Heisenberg bundle is quite irregular in any classical sense.

But we shall see in  this paper  that, nevertheless, 
all NCP $\TT^n$-bundles are locally trivial in a $K$-theoretical sense.
This shows up if we change from the category of C*-algebra bundles over $X$ 
with bundle preserving (i.e., $C_0(X)$-linear)
$*$-homomorphisms to the category $\RKK_X$ of C*-algebra bundles over $X$ 
with morphisms given by the elements of Kasparov's group $\RKK(X; A(X), B(X))$.
Isomorphic  bundles $A(X)$ and $B(X)$ in this category are precisely the $\RKK$-equivalent
bundles.
The first observation we can make is the following

\begin{thm}\label{thm-local}
Any NCP $\TT^n$-bundle $A(X)$ is locally $\RKK$-trivial.
This means that for every $x\in X$ there is a neighbourhood $U_x$ of $x$ such that
the restriction $A(U_x)$ of $A(X)$ to $U_x$ is $\RKK$-equivalent 
to the trivial bundle $C_0(U_x, A_x)$.
\end{thm}
 
 The proof is given in Corollary \ref{cor-localRKK-triv} below.
 Having this result, it is natural to ask the following questions:
 \\
 \\
 {\bf Question 1:} Suppose that $A_0(X)$ and $A_1(X)$ are two non-commutative $\TT^n$-bundles 
 over $X$. Under what conditions is $A_1(X)$ $\RKK$-equivalent to $A_2(X)$?
 \\
 
 Actually,  in this paper 
 we will only give a partial answer to the above question. 
 But we shall  give a complete answer, at least for (locally) path connected spaces $X$, to 
 \\
 \\
 {\bf Question 2:} Which non-commutative principal $\TT^n$-bundles are $\RKK$-equivalent 
 to a ``commutative'' $\TT^n$-bundle?
\\

A basic tool for our  study of $\RKK$-equivalence of non-commutative  torus bundles will
be the $K$-theory group bundle $\mathcal K_*:=\{K_*(A_x): x\in X\}$ associated
to $A(X)$ and/or a certain associated action of the fundamental group $\pi_1(X)$ on the fibres 
$K_*(A_x)$  if $X$ is path-connected.  
The answer to Question 2 is then given by

\begin{thm}\label{thm-trivil}
Let $A(X)$ be a NCP $\TT^n$-bundle with $X$ path connected. Then the following are equivalent:
\begin{enumerate}
\item $A(X)$ is $\RKK$-equivalent to a commutative principal $\TT^n$-bundle.
\item The $K$-theory bundle $\mathcal K_*(A(X))$ is trivial.
\item The action of  $\pi_1(X)$ on the fibres $K_*(A_x)$ of the $K$-theory bundle is trivial.
\end {enumerate}
\end{thm}

In the case where $n=2$ we can determine the $\RKK$-equivalence classes up to a twisting 
by commutative principal bundles and we shall obtain a slightly weaker result in case $n=3$
(see Theorem \ref{thm-RKK} below). But this still does not give a complete answer to 
Question 1. 


Another obvious obstruction for $\RKK$-equivalence is
given by the $K$-theory group of the ``total space'' $A(X)$:
 If $K_*(A_0(X))\neq K_*(A_1(X))$, then 
clearly $A_0(X)$ and $A_1(X)$ are not $\RKK$-equivalent, since $\RKK$-equivalence implies
$\KK$-equivalence. Using this obstruction we can easily see that the 
Hopf-fibration $p:S^3\to S^2$ and the trivial bundle $S^2\times \TT$ are 
principal $\TT$-bundles which are not $\RKK$-equivalent (i.e., $C(S^3)$ is not 
$\RKK(S^2;\cdot,\cdot)$-equivalent to $C(S^2\times\TT)$), since $K_0(C(S^3))=\ZZ$
while $K_0(C(S^2\times \TT))=\ZZ^2$, although, $S^2$ being simply
connected, their $K$-theory group
bundles are trivial and thus isomorphic. However, in general it is not so easy to 
compute the $K$-theory groups of the algebras $A(X)$! Since, as mentioned above, our bundles 
behave like Serre fibrations in commutative topology, one should 
expect the existence of a generalized Leray Serre spectral sequence which relates
the $K$-theory group  of the total space $A(X)$ to the $K$-theory of the base and the fibres.
Indeed, such spectral sequence does exist with $E_2$-term given by
the cohomology of $X$  with local coefficients in the $K$-theory group bundle
$\mathcal K_*(A(X))$, but we postpone the details of this 
to the forthcoming paper \cite{ENO}.

This paper is organized as follows: after a preliminary section on C*-algebra bundles we
start in \S \ref{sec-torus} with the classification of NCP $\TT^n$-bundles based on  
\cite{EW1, EW2}. We then proceed in \S \ref{sec-loctriv} by showing that all NCP bundles
are locally $\RKK$-trivial before we introduce the $K$-theory group bundle and 
the action $\pi_1(X)$ on the fibres of such bundles in \S \ref{sec-bundle}. 
This action will be determined in detail for the case $n=2$ in \S \ref{sec-dim2}, and,
building on the two-dimensional case the $\pi_1(X)$-action will be described 
in  general case in \S \ref{sec-general}. In \S \ref{RKK} we will then give our main results 
on $\RKK$-equivalence of NCP bundles as explained above.

 In our final section, 
\S \ref{sec-T-dual} we give some application of our results to the study of 
$T$-duals of principal $\TT^n$-bundles $q:Y\to X$ with $H$-flux $\delta\in H^3(Y,\ZZ)$
as studied by Mathai and Rosenberg in \cite{MR1,MR2}. The corresponding 
stable continuous-trace algebras $CT(Y,\delta)$ can then regarded as C*-algebra bundles
over $X$ with fibres isomorphic to $C(\TT^n, \K)$ (it follows from the fact that 
$CT(Y,\delta)$ is assumed to carry an action of $\RR^n$ which restricts on the base $Y$
to the action inflated from  the given action of $\TT^n$ on $Y$, that the Dixmier-Douady class $\delta$
is trivial on the fibres $Y_x\cong \TT^n$). 
We can then study the $K$-theory group bundle
$\mathcal K_*(CT(Y,\delta))$ over $X$ with fibres $K_*(C(Y_x,\K))\cong K_*(C(\TT^n))$ for $x\in X$.
Following \cite{MR1,MR2} we say that $(Y,\delta)$ has a classical $T$-dual $(\widehat{Y},\widehat{\delta})$ if and only if one can find an action $\beta:\RR^n\to \Aut(CT(Y,\delta))$ as above, such that 
the crossed product $CT(Y,\delta)\rtimes_\beta\RR^n$ is again a continuous-trace algebra,
and hence isomorphic to some $CT(\widehat{Y}, \widehat{\delta})$ for some principal
$\TT^n$-bundle $\hat{q}:\widehat{Y}\to X$ and some $H$-flux 
$\widehat{\delta}\in H^3(\widehat{Y},\ZZ)$. Combining our results with \cite[Theorems 2.3 and 3.1]{MR2}
we show

\begin{thm} Suppose that $X$ is (locally) path connected.
Then the  pair $(Y,\delta)$ has a classical $T$-dual $(\widehat{Y},\widehat{\delta})$ 
if and only if the associated  $K$-theory group bundle
$\mathcal K_*(CT(Y,\delta))$ over $X$ is trivial.
\end{thm}


\section{Preliminaries on C*-algebra bundles.}

In this section we want to set up some basic results on C*-algebra bundles 
which are used throughout this paper. As explained in the introduction, we 
use the term C*-algebra bundle as a more suggestive word for 
what is also known as a $C_0(X)$-algebra in the literature, and we already recalled
the definition of these objects in the introduction. 
Recall that for $x\in X$ the fibre $A_x$ of $A(X)$ at $x$ was defined as the 
quotient $A_x=A(X)/I_x$ with $I_x=C_0(X\setminus \{x\})A(X)$ (throughout the paper
we shall simply  write $g\cdot a$ for $\Phi(g)a$ if $g\in C_0(X)$ and $a\in A$).
If $a\in A(X)$, we put $a(x):=a+I_x\in A_x$. In this way we may view the elements 
$a\in A(X)$ as sections of the bundle $\{A_x: x\in X\}$. The function $x\mapsto \|a(x)\|$
is then always upper semi-continuous and vanishes at infinity on $X$. Moreover,
we have 
$$\|a\|=\sup_{x\in X}\|a(x)\|\quad\text{for all $a\in A(X)$}.$$
This and many more details can be found in \cite[Appendix C]{Will}.

If $A(X)$ is a C*-algebra bundle, $Y$ is a locally compact 
space and $f:Y\to X$ is a continuous map, then the pull-back $f^*A$ of
$A(X)$ along $f$ is defined as the balanced tensor product
$$f^*A:=C_0(Y)\otimes_{C_0(X)}A,$$
where the $C_0(X)$-structure on $C_0(Y)$
is given via $g\mapsto g\circ f\in C_b(Y)=M(C_0(Y))$.
The $*$-homomorphism 
$$C_0(Y)\to C_0(Y)\otimes_{C_0(X)}A; g\mapsto g\otimes 1$$
provides $f^*A$ with a canonical structure as a C*-algebra bundle over $Y$
with fibre $f^*A_y=A_{f(y)}$ for all $y\in Y$. We shall therefore use the notation
$f^*A(Y)$ to indicate this structure.

In particular, if $Z\subseteq X$ is a locally compact subset of $X$ 
the pull-back $A(Z):=i_Z^*A$ of $A$ along the inclusion
map $i_Z:Z\to X$ becomes a C*-algebra bundle over $Z$ which we call the 
{\em restriction} of $A(X)$ to $Z$. The following lemma is well known. 
A proof can be found in \cite{EW2}.

\begin{lem}\label{lem-pull1} Suppose that $A(X)$ is a C*-algebra bundle 
and let $Z$ be a closed subset
of $X$ with complement $U:=X\smallsetminus Z$. Then there is a canonical short exact sequence
of $C^*$-algebras
$$0\to A(U)\to A(X)\to A(Z)\to 0.$$
The quotient map $A\to A(Z)$ is given by restriction  of sections in $A(X)$ 
 to $Z$ and the inclusion $A(U) \subseteq A(X)$ is given via the $*$-homomorphism
$$C_0(U)\otimes_{C_0(X)}A(X)\to A(X); g\otimes a\mapsto g\cdot a$$
(which makes sense since $C_0(U)\subseteq C_0(X)$).
\end{lem}

\begin{remark}\label{rem-loctriv}
It might be useful to remark the following relation between taking pull-backs and
restricting to subspaces: assume that $A$ is a $C_0(X)$-algebra and $f:Y\to X$ is a 
continuous map. Then the tensor product $C_0(Y)\otimes A(X)$ becomes 
a C*-algebra bundle over $Y\times X$
via the the canonical imbedding of $C_0(Y\times X)\cong C_0(Y)\otimes C_0(X)$ into 
$ZM(C_0(Y)\otimes A)$. The pull-back $f^*A(Y)$ then coincides with the restriction 
of $C_0(Y)\otimes A(X)$ to the graph $Y\cong \{(y,f(y)): y\in Y\}\subseteq Y\times X$.
\end{remark}

Recall that a continuous map $f:Y\to X$ between two locally compact spaces is called {\em proper}
if the inverse images of compact sets in $X$ are compact. If $f:Y\to X$ is proper,
then $f$ induces a non-degenerate $*$-homomorphism 
$\phi_f:C_0(X)\to C_0(Y); g\mapsto g\circ f$. We shall now extend this observation to 
C*-algebra bundles.

\begin{lem}\label{lem-pull2}
Assume that $A(X)$ is a C*-algebra bundle and that $f:Y\to X$ is a
proper map.
Then,
\begin{enumerate}
\item for any $a$ in $A$, the element $1\otimes a$ of  $M(f^*A(Y))$
  actually 
  lies in $f^*A(Y)=C_0(Y)\otimes_{C_0(X)}A(X)$;
\item  there exists a well defined $*$-homomorphism 
$\Phi_f: A(X)\to f^*A(Y)$;
given by $\Phi_f(a)=1\otimes a$. 
\item If $g:Z\to Y$ is another proper map, then
$\Phi_{f\circ g}= \Phi_g\circ \Phi_f$. 
\item In particular,
if $f$ is a homeomorphism, then $\Phi_f$ is an isomorphism.
\end{enumerate}
\end{lem}
\begin{proof} Notice that for any continuous map $f:Y\to X$
the map $\Phi_f:A(X)\to M(f^*A(Y))$ given by $\Phi_f(a)=1\otimes a$ is a well defined 
$*$-homomorphism.
Thus we only have to show that, in case where $f$ is proper, the image lies in $f^*A(Y)$.
Note that via the above map each $a\in A(X)$ is mapped to the section $\tilde{a}$ 
of the bundle $f^*A(Y)$ with $\tilde{a}(y)=a(f(y))$. If $f$ is proper, it follows that 
the norm function $y\mapsto \|\tilde{a}(y)\|=\|a(f(y))\|$ vanishes at $\infty$.
Now choose a net of compacts $\{C_i\}_{\i\in I}$ 
in $Y$ and continuous compactly supported functions $\chi_i:Y\to[0,1]$ with $\chi_i|_{C_i}\equiv 1$.
Then the image of $\chi_i\otimes a$ in $f^*A(Y)$ will converge to $\tilde{a}$ in norm, which 
clearly implies that $\tilde{a}\in f^*A(Y)$. All other conditions now follow from the fact that
$(g\circ f)^*A(Z)$ is canonically isomorphic to $g^*(f^*A)(Z)$.
\end{proof}

\section{Classification of non-commutative torus bundles}\label{sec-torus}

As noted already in the introduction, the NCP torus bundles $A(X)$
can be realized up to $\TT^n$-equivariant Morita equivalence over $X$ 
as crossed products $C_0(X,\K)\rtimes \ZZ^n$,
equipped with the dual $\TT^n$-action, 
where $\ZZ^n$ acts fibre-wise on the trivial bundle $C_0(X,\K)$.
For the necessary background on  equivariant Morita equivalence with respect to 
dual actions of abelian groups we refer to \cite{E-memoir, EKQR}, and for 
the necessary background on  crossed products by fibre-wise actions we refer
to \cite{EW1, EW2} or \cite{KW}. We should then make the translation 
of NCP torus bundles to $\ZZ^n$-actions more precise:

\begin{prop}\label{prop-dualbundle}
Every NCP $\TT^n$-bundle $A(X)$ is $\TT^n$-equivariantly
Morita equivalent over $X$ to some crossed product $C_0(X,\K)\rtimes \ZZ^n$, where 
$\ZZ^n$ acts  fibre-wise on $C_0(X,\K)$, and where the $\TT^n$-action
on $C_0(X,\K)\rtimes \ZZ^n$ is the dual action. 

Moreover, two bundles $A(X)=C_0(X,\K)\rtimes_{\alpha}\ZZ^n$ and $B(X)=C_0(X,\K)\rtimes_{\beta}\ZZ^n$ are $\TT^n$-equivariantly Morita equivalent over $X$ if and only if 
the two actions $\alpha,\beta: \ZZ^n\to C_0(X,\K)$ are Morita equivalent over $X$.
\end{prop}

Using the above Proposition  together with the results of \cite{EW1,EW2}
we can give a complete classification of NCP torus bundles
up to equivariant Morita equivalence over $X$.
In general, if $G$ is a (second countable) locally compact group, it is shown in 
\cite[Proposition 5.1]{EW1} that 
two fibre-wise actions $\alpha$ and $\beta$ on $C_0(X,\K)$ are $G$-equivariantly Morita equivalent over $X$
if and only if they are {\em exterior equivalent}, i.e., there exists a continuous 
map $v:G\to UM(C_0(X,\K))\cong C(X, \U)$ such that 
$$\alpha_s=\Ad v_s\circ \beta_s\quad\text{and}\quad 
v_{st}=v_s\beta_s(v_t)$$
for all $s,t\in G$, where $\U=\U(l^2(\NN))$ denotes the unitary group.
Note that exterior equivalent actions have 
isomorphic crossed products (e.g. see  \cite{Ped}).
It is shown in \cite{EW1} (following the ideas of \cite{CKRW}), that 
the set $\E_G(X)$ of all exterior equivalence classes $[\alpha]$ of 
fibre-wise actions $\alpha:G\to \Aut_{C_0(X)}(C_0(X,\K))$ forms a group with 
multiplication given by 
$$[\alpha][\beta]=[\alpha\otimes_{C_0(X)}\beta]$$
where $\alpha\otimes_{C_0(X)}\beta$ denotes the diagonal action on 
$$C_0(X,\K)\otimes_{C_0(X)}C_0(X,\K)\cong C_0(X,\K\otimes \K)\cong C_0(X,\K).$$
Hence, it follows from Proposition \ref{prop-dualbundle} that the problem of classifying 
NCP $\TT^n$-bundles over $X$ up to equivariant Morita equivalence is
equivalent to the problem of describing the group $\E_{\ZZ^n}(X)$!

We first consider the case $X=\{\pt\}$, i.e., the classification of group actions on $\K$ 
up to exterior equivalence. Let $\PU=\U/\TT1$ denote the projective unitary group.
Then $\PU\cong \Aut(\K)$ via $[V]\mapsto \Ad V$ and an action of $G$ on $\K$
 is a strongly continuous homomorphism 
$\alpha:G\to \PU$. Choose a Borel lift $V:G\to \U$ for $\alpha$. It is then  easy  to check that
$$V_sV_t=\om_{\alpha}(s,t)V_{st}$$
for some Borel group-cocycle $\om_{\alpha}\in Z^2(G,\TT)$. One can then show that
$$\E_G(\{\pt\})\to H^2(G,\TT);[\alpha]\mapsto [\om_{\alpha}]$$
is an isomorphism of groups.
To describe the crossed product $\K\rtimes_{\alpha}G$ in terms of
 $[\om_{\alpha}]$ we have to 
recall the construction of the {\em twisted group algebra} 
$C^*(G,\om_{\alpha})$: It is defined as the enveloping $C^*$-algebra of the 
twisted $L^1$-algebra $L^1(G,\om_{\alpha})$, which is
defined as the  Banach space $L^1(G)$ with twisted convolution/involution
 given by the formulas
$$f*_{\om_{\alpha}}g(s)=\int_G f(t)g(t^{-1}s)\om_{\alpha}(t, t^{-1}s)\,dt
\quad\text{and}\quad
f^*(s)=\Delta(s^{-1})\overline{\om_{\alpha}(s, s^{-1})f(s^{-1})}.$$
The crossed product $\K\rtimes_{\alpha}G$ is then isomorphic to $\K\otimes C^*(G,\bar{\om}_{\alpha})$
(e.g. see \cite[Theorem 1.4.15]{E-memoir}), where $\bar{\om}_{\alpha}$ denotes the complex conjugate of ${\om}_{\alpha}$.
Notice that up to canonical isomorphism, $C^*(G,\om_{\alpha})$ only depends on the cohomology 
class of $\om_{\alpha}$ in $H^2(G,\TT)$. 

If $G=\ZZ^n$, then one can show that 
any cocycle $\om:\ZZ^n\times \ZZ^n\to \TT$ is equivalent to one of the form 
\begin{equation}\label{eq-cocycleZn}
(n,m)\mapsto 
\om_{\Theta}(n,m)=e^{2\pi i\lk \Theta n,m\rk},
\end{equation}
where $\Theta\in M(n\times n,\RR)$ 
is a  strictly upper triangular matrix. 
If we denote by $u_1,\ldots, u_n\in L^1(G,\om_{\Theta})$ the Dirac functions of the unit vectors 
$e_1,\ldots, e_n\in \ZZ^n$,
we can check that the twisted group algebra $C^*(\ZZ^n, \om_{\Theta})$ coincides with the universal
C*-algebra generated by $n$ unitaries $u_1, \ldots , u_n$ satisfying the relations
$$u_ju_i =e^{2\pi i\Theta_{ij}}u_i u_j$$
and therefore coincide with what is usually considered as a non-commutative $n$-torus!

Suppose now that $\alpha:G\to \Aut(C_0(X,\K))$ is any fibre-wise action.
We then obtain a map $f_{\alpha}:X\to H^2(G,\TT)$ given by the composition
 $$
 \begin{CD}
 X  @>x\mapsto [\alpha^x]>> \E_G(\{\pt\})  @>[\alpha^x]\mapsto  [\om_{\alpha^x}]>> H^2(G,\TT),
\end{CD}
 $$
and it is shown in \cite[Lemma 5.3]{EW1} that this map is actually continuous with respect to the 
Moore topology on $H^2(G,\TT)$ as introduced in \cite{mo3}. Thus we obtain 
a homomorphism of groups $\Phi: \E_G(X)\to C(X, H^2(G,\TT)); \alpha\mapsto f_{\alpha}$. If 
 $G$ is abelian and  compactly generated, then the kernel of this map
is isomorphic to the group $\check{H}^1(X, \widehat{G})$, which 
classifies the principal $\widehat{G}$-bundles over $X$.
If $q:Y\to X$ is such a bundle,
the corresponding action $\alpha_Y: G\to \Aut(C_0(X,\K))$ is just
the dual action  of $G$ on $C_0(Y)\rtimes \widehat{G}\cong C_0(X,\K)$ -- you should 
compare this with the discussion in the introduction. 
We thus obtain an exact sequence
$$0\to \check{H}^1(X,\widehat{G})\to \E_G(X)\stackrel{\Phi}{\to} C(X,H^2(G,\TT)).$$
Recall further that if  $1\to Z\to H\to G\to 1$ is a locally compact central extension of $G$
and if $c:G\to H$ is any Borel section, then
$$\eta: G\times G\to Z; \eta(s,t)=c(s)c(t) c(st)^{-1}$$
is a cocycle in $Z^2(G,Z)$ which represents the given extension in $H^2(G,Z)$.
The  {\em transgression map}
$$\tg :\widehat{Z}\to H^2(G,\TT);\chi\mapsto [\chi\circ \eta]$$
is defined by composition of $\eta$ with the characters of $Z$.
The extension 
 $1\to Z\to H\to G\to 1$ is called
 a {\em representation group} for $G$ (following \cite{mo2}) if $\tg:\widehat{Z}\to H^2(G,\TT)$ is an 
 isomorphism of  groups. It is shown in \cite[Corollary 4.6]{EW1} 
 that such a representation group exists for
 all compactly generated abelian groups.
 Using the obvious homomorphism $Z\to C(\widehat{Z},\TT)$
and composing it with $\eta$ we obtain a cocycle (also denoted $\eta$) in $Z^2(G, C(\widehat{Z},\TT))$
which determines a fibre-wise action $\mu:G\to \Aut(C_0(\widehat{Z},\K))$
as described in \cite[Theorem 5.4]{EW1}. Identifying $C(X, H^2(G,\TT))$ with 
$C(X,\widehat{Z})$ via the trangression map, we then obtain a splitting homomorphism
$$F: C(X,\widehat{Z})\to \E_G(X);  f\mapsto f^*(\mu),$$
where $f^*(\mu)$ denotes the {\em pull-back of $\mu$ to $X$ via $f$}, which is defined via the
diagonal action
$$f^*(\mu)=\id\otimes_{C(\widehat{Z})}\mu:G\to \Aut\big(C_0(X)\otimes_{C_0(\widehat{Z}),f}C_0(\widehat{Z},\K)\cong C_0(X,\K)\big).$$
Notice that the homomorphism $F$ depends on the choice of the representation group 
$H$ and is therefore not canonical. However, for $G=\ZZ^n$ it is unique up to isomorphism of groups
by \cite[Proposition 4.8]{EW1}. 
We can now combine all this in the following statement 
(note that a similar result holds for many non-abelian groups by \cite[Theorem 5.4]{EW1}
and \cite{EN}).

\begin{thm}[{cf \cite[Theorem 5.4]{EW1}}]\label{thm-class}
Suppose that $G$ is a compactly generated second countable abelian group 
with fixed representation group $0\to Z\to H\to G\to 1$.
Let $X$ be
a second countable locally compact space. Then the exterior equivalence 
classes of fibre-wise actions of $G$ 
on $C_0(X,\K)$ are classified by all pairs $([q:Y\to X], f)$, where 
$[q:Y\to X]$ denotes the isomorphism class of a principal $\widehat{G}$-bundle $q:Y\to X$
over $X$, and $f\in C(X, \widehat{Z})$.
\end{thm}

The above classification also leads to a direct description of the corresponding 
crossed products $C_0(X,\K)\rtimes G$ in termes of the given topological data: It is shown in 
\cite{EW2} that if $[\alpha]$ corresponds to the pair $([q:Y\to X], f)$ 
then 
$$C_0(X,\K)\rtimes G\cong Y*\big(f^*\big(C^*(H)\big)(X)\big)\otimes \K,$$
by a $\widehat{G}$-equivariant bundle isomorphism over $X$.
Here $f^*\big(C^*(H)\big)(X)$ is the pull-back of $C^*(H)(\widehat{Z})$, where
the $\widehat{Z}$-C*-algebra bundle structure of $C^*(H)$ is given by the structure map
$$C_0(\widehat{Z})\cong C^*(Z)\to ZM(C^*(H)),$$
with action of $C^*(Z)$ on $C^*(H)$  given by convolution (see \cite[Lemma 6.3]{EW2}
for more details).
The product $Y*\big(f^*C^*(H))(X)$ of the principal $\widehat{G}$-bundle $q:Y\to X$
with $f^*(C^*(H))(X)$ is described carefully in \cite[\S 3]{EW2} and is 
a generalization of the usual fibre product
$$Y*Z:=(Y\times_XZ)/{\widehat{G}}$$
if $\widehat{G}$ acts fibre-wise on the fibred space $p:Z\to X$, and 
 diagonally on $Y\times_XZ=\{(y,z)\in Y\times Z: q(y)=p(z)\}$.
In this paper we only need the construction in case of $\widehat{G}=\TT^n$ being compact, 
in which case the fibre product $Y*A(X)$ can be defined as the fixed point algebra
$$q^*A(Y)^G= \big(C_0(Y)\otimes_qA(X)\big)^G$$
under the diagonal action $g(\varphi\otimes a)=g^{-1}(\varphi)\otimes\alpha_g(a)$ for $g\in G$,
$\varphi\in C_0(Y)$ and $a\in A(X)$. The $G$-action on $Y*A(X)$ is then induced by the given 
 $G$-action on $C_0(Y)$, i.e., we have
 $$(Y*\alpha)_g(\varphi\otimes a)=g(\varphi)\otimes a=\varphi\otimes\alpha_g(a).$$

We now specialize to the group $G=\ZZ^n$. It is shown in \cite[Example 4.7]{EW1} 
that a representation group
for $\ZZ^n$ can be constructed as the central group extension
$$1\to Z_n\to H_n\to \ZZ^n\to 1$$
where $Z_n$ is the additive  group of strictly upper triangular integer matrices
and $H_n=Z_n\times \ZZ^n$ with multiplication given by
$$(M,\m)\cdot(K,\l)=(M+K+ \eta(\m,\l), \m+\l)\quad\text{where} \quad\eta(\m,\l)_{ij}=l_im_j.$$
for $\m=(m_1,\ldots, m_n), \l=(l_1,\ldots, l_n)\in \ZZ^n$. 
Let $T_n:=\widehat{Z}_n$ denote the dual group of $Z_n$. Since $Z_n$ is canonically isomorphic to 
$\ZZ^{\frac{n(n-1)}{2}}$, we have $T_n\cong\TT^{\frac{n(n-1)}{2}}$. 
If we express a  character $\chi\in T_n$ 
by a real upper triangular matrix $\Theta=(a_{i,j})_{1\leq i,j\leq n}$ via 
$$\chi_\Theta(M)=\prod_{i<j} (e^{2\pi i a_{ij}})^{m_{ij}}=\exp\big( 2\pi i\sum_{i<j} a_{ij}m_{ij}\big),$$
we see that the transgression map $\tg: T_n\to H^2(\ZZ^n,\TT);\chi\mapsto[\chi\circ \eta]$ sends 
$ \chi_{\Theta}$ to the class of the cocycle $\om_{\Theta}$ as defined
in (\ref{eq-cocycleZn}) and is an isomorphism of groups.

As a consequence of the above discussion, we see that 
$C^*(H_n)$  serves as a ``universal'' NCP $\TT^n$-bundle 
over $\TT^{\frac{n(n-1)}{2}}$ (where the action of $\TT^n$ is the canonical dual action 
of $\TT^n=\widehat{\ZZ^n}$ on $C^*(H_n)$). 
We finally arrive at the following classification of non-commutative $\TT^n$-bundles
up to $\TT^n$-equivariant Morita equivalence over $X$:

\begin{thm}\label{thm-T-bundle}
Let $X$ be a second countable locally compact space. Then the set of 
$\TT^n$-equivariant Morita equivalence classes  of NCP $\TT^n$-bundles over $X$
is classified by the set of all pairs
$([q:Y\to X], f)$ with $q:Y\to X$ a (commutative) principal $\TT^n$-bundle over $X$ and 
 $f:X\to \TT^{\frac{n(n-1)}{2}}$ a continuous map. Given these data, the 
  corresponding equivalence class of NCP $\TT^n$-bundles
is represented by the algebra
$$A_{(Y,f)}(X):=Y* f^*\big(C^*(H_n)\big)(X).$$
\end{thm}

\begin{remark}\label{rem-Heisenberg}
Notice that the group $H_2$ is just the discrete Heisenberg group. 
Therefore, the C*-group algebra $C^*(H_2)$ of the Heisenberg group serves a 
a universal NCP $\TT^2$-bundle in the sense explained above.
\end{remark}

There is a natural action of $GL_n(\ZZ)$ on the set of equivalence classes of NCP $\TT^n$-bundles
over $X$ which we want to describe in terms of the above classification. 
For this we first introduce some notation:
If $A(X)$ is an NCP $\TT^n$-bundle over $X$ and if $\Psi\in GL_n(\ZZ)$, we denote by
$A_{\Psi}(X)$ the NCP $\TT^n$-bundle which we obtain from $A(X)$ by composing the 
given action of $\TT^n$ by $\Psi^{-1}$ (viewed as an automorphism of $\TT^n$).
Of course, if we start with a commutative principal $\TT^n$-bundle $q:Y\to X$, we obtain 
in the same way a principal $\TT^n$-bundle $q_{\Psi}:Y_{\Psi}\to X$.

If we denote for $1\leq i,j\leq n$ by $e_{i,j}$ the elementary matrix
of $M_n(\ZZ)$ (with $1$ in  the $i$-th row and $j$-th column entry and
$0$ every else) and by $f_1,\ldots,f_n$ the canonical basis of
$\ZZ^n$, then we have for $1\leq i<j\leq n$  the following  relations in $H_n$:
\begin{itemize}
\item the elements $(e_{i,j},0)$   are central;
\item $(0,f_i)(0,f_j)=(e_{i,j},0)(0,f_j)(0,f_i)$.
\end{itemize}
Moreover, these generators and relations define  the group $H_n$:
\begin{prop}
The group $H_n$ is the group defined by the generators
$x_1,\ldots,x_n$ and $\{y_{i,j};\, 1\leq i<j\leq n\}$ and the following
relations  for $1\leq i<j\leq n$:
\begin{itemize}
\item the element $y_{i,j}$ is central;
\item $x_i\cdot x_j=y_{i,j}\cdot x_j\cdot x_i$.
\end{itemize}

\end{prop}
\begin{proof}Let $H'_n$ be the group defined by the above generators and
relations. Then we have an obvious epimorphism from $H'_n\to H_n$. From  the universal property of $H'_n$, we also get an
epimorphism $\lambda:H'_n\to \ZZ^n$ such that the following diagram commutes
$$
\begin{CD}H'_n @>\lambda>>\ZZ^n\\
@VVV      @V\parallel VV\\
H_n@>>>\ZZ^n.
\end{CD}
$$
The left vertical arrow gives rise to an epimorphism $\ker \lambda\to Z_n$, splitted by
$e_{i,j}\mapsto x_{i,j}$. In fact, using the commutation rules of $H'_n$, any
word $w$ can be writen $x_1^{k_1}\cdots x_n^{k_n}\cdot w'$, where
$k_1,\ldots,k_n$ belongs to $\ZZ$ and where $w'$ is a word with letters
in $\{y_{i,j};\, 1\leq i<j\leq n\}$. We then have $\lambda(w)=(k_1,\ldots,k_n)$.
It follows that
 $\ker \lambda$ is generated by $\{y_{i,j};\, 1\leq i<j\leq n\}$. Thus, the epimorphism
$H'_n\to H_n$ induces by restriction an isomorphism  $\ker
\lambda\stackrel{\cong}{\to} Z_n$. The result then follows from the Five Lemma.
\end{proof}
\begin{cor}\label{cor-universal}
The $C^*$-algebra $C^*(H_n)$ is the universal $C^*$-algebra generated
by unitaries
$U_1,\ldots,U_n$ and $\{V_{i,j};\, 1\leq i<j\leq n\}$ and the following
relations  for $1\leq i<j\leq n$:
\begin{itemize}
\item the elements  $V_{i,j}$ are central;
\item $U_i\cdot U_j=V_{i,j}\cdot U_j\cdot U_i$.
\end{itemize}
\end{cor} 
The unitaries  $V_{i,j}$ for $1\leq i<j\leq n$ generate a copy of
$C(\TT^{\frac{n(n-1)}{2}})$ providing the $C^*$-algebra bundle
structure of  $C^*(H_n)$ over $\TT^{\frac{n(n-1)}{2}}\cong
\Lambda_2(\RR^n)/\Lambda_2(\ZZ^n)$.
Under this identification, $V_{i,j}=e^{2\pi i f^*_i\wedge f^*_j}$, where
$(f^*_1,\ldots,f^*_n)$ is the dual basis to the canonical basis
$(f_1,\ldots,f_n)$ of $\RR^n$. This  allows to extend the definition
of the unitaries $V_{i,j}$ for values 
  $1\leq j\leq i\leq n$ by setting  $V_{i,j}:=e^{2\pi i f^*_i\wedge f^*_j}$. Then
$U_i\cdot U_j=V_{i,j}\cdot U_j\cdot U_i$ for all $1\leq i,j \leq
n$. If we denote by $\beta^n$ the
canonical dual action of $\TT^n$ on $C^*(H_n)$, then
$\beta^n_{(z_1,\ldots,z_n)}(U_k)=z_kU_k$.

\medskip

Recall that if  $A(X)$ is a  NCP $\TT^n$-bundle with $\TT^n$-action 
$\alpha:\TT^n\to \Aut(A(X))$ and if $\Psi$ is any  matrix of
$GL_n(\ZZ)$, then  $A_{\Psi}(X)$ is defined as the NCP $\TT^n$-bundle with respect to
$\alpha\circ \Psi^{-1}$.  

\begin{lem}\label{lem-aut-heis}
Let $\Psi$ be a matrix of  $GL_n(\ZZ)$ and let $\Lambda_2\Psi$ be the
automorphism of $\Lambda_2(\ZZ^n)$ induced by  $\Psi$. Then using the
identification $\TT^{\frac{n(n-1)}{2}}\cong
\Lambda_2(\RR^n)/\Lambda_2(\ZZ^n)$, we have an automorphism
$\Upsilon_\Psi:C^*(H_n)\to C^*(H_n)$ such that
\begin{enumerate}
\item  $\Upsilon_\Psi(f\cdot x)=f\circ \Lambda_2\Psi^{-1}\cdot \Upsilon_\Psi(x)$ for all
  $f$ in $C(\TT^{\frac{n(n-1)}{2}})$ and all $x$ in $C^*(H_n)$.
\item $\Upsilon_\Psi(\beta^n_z\cdot
  x)=\beta^n_{\Psi(z)}\cdot\Upsilon_\Psi(x)$, for all $z$ in $\TT^n$ and all
  $x$ in $C^*(H_n)$.
\end{enumerate}
\end{lem}
\begin{proof}
Notice first  that $V_{i,j}\circ \Lambda_2\Psi^{-1}=e^{2\pi i ^t\Psi^{-1}
  f^*_i\wedge^t\Psi ^{-1}f^*_j}$.  
\begin{itemize}
\item We first prove the result for $\Psi$ given by the elementary matrix
  $E_{k,l}(m)=I_n+me_{k,l}$ of $GL_n(\ZZ)$ where $k\neq l$ and $m$ is
  in $\ZZ$. Then, 
  for $i\neq j$ and $k\neq l$, we have $E_{k,l}(m)^{-1}=E_{k,l}(-m)$ and
\begin{eqnarray}\label{eq-gen-1}
V_{i,j}\circ \Lambda_2 E_{k,l}(m)^{-1}&=&V_{i,j}\quad\text{ if
}k\notin\{i,j\}\text{ or } l\in\{i,j\} \\\label{eq-gen-2}
V_{j,k}\circ \Lambda_2 E_{k,l}(m)^{-1}&=&V_{j,k}V_{j,l}^{-m}\quad\text{ if
} j\neq k\text{ and } j\neq l.
\end{eqnarray}
There is a unique
morphism
$\Upsilon_{E_{k,l}(m)}:C^*(H_n)\to C^*(H_n)$ with image on generators
\begin{eqnarray*}
U_i&\mapsto&U_i\quad\text{ if } i\neq k\\
U_k&\mapsto&U_l^*U_k\\
V_{i,j}&\mapsto&V_{i,j}\quad\text{ if
}k\notin\{i,j\}\text{ or } l\in\{i,j\}\\
V_{j,k}&\mapsto&V_{j,k}V_{j,l}^*\quad\text{ if
} j\neq k\text{ and } j\neq l,
\end{eqnarray*} since  it is straightforward to check that the
operators on the right
hand side satisfy the relation as given in Corollary
\ref{cor-universal}. Moreover, $\Upsilon_{E_{k,l}(m)}$ is an
isomorphism with
$\Upsilon_{E_{k,l}(m)}^{-1}=\Upsilon_{E_{k,l}(-m)}$. It is enought to
check condition (i) on the generators $V_{i,j}$, which is done
in equations (\ref{eq-gen-1}) and (\ref{eq-gen-2}). Condition (ii) has to
be checked only on the generators $U_i$. If $z=(z_1,\ldots,z_n)$, then
the $i$-th component of 
$E_{k,l}(m)(z)$ is $z_i$ for  $i\neq k$ and $z_kz_l^m$ for $i=k$, and 
$$\Upsilon_\Psi(\beta^n_z\cdot
  U_i)=z_iU_i=\beta^n_{\Psi(z)}\cdot\Upsilon_\Psi(U_i)$$ for $i\neq k$
  and $$\Upsilon_\Psi(\beta^n_z\cdot
  U_k)=z_kU^{-m}_l\cdot U_k=(z_l\cdot U_l)^{-m}\cdot z_kz_l^mU_k=\beta^n_{\Psi(z)}\cdot\Upsilon_\Psi(U_k).$$
\item Let $\sigma$ be a permutation on $\{1,\ldots,n\}$ and let
$\Psi_\sigma=(\delta_{i,\sigma(j)})_{1\leq i,j\leq n}$ be the
permutation matrix of $GL_n(\ZZ)$ corresponding to $\sigma$. As
before, there is an automorphism 
 $\Upsilon_{\Psi_\sigma}$   of 
$C^*(H_n)$ uniqually defined  on generators by $U_i\mapsto U_{\sigma^{-1}(i)}$ and 
$V_{i,j}\mapsto V_{\sigma^{-1}(i),\sigma^{-1}(j)}$. Since
$$V_{\sigma^{-1}(i),\sigma^{-1}(j)}=V_{i,j}\circ \Lambda_2\Psi_{\sigma^{-1}}=V_{i,j}\circ
\Lambda_2\Psi_\sigma^{-1}$$ and
$$\beta^n_{\Psi_\sigma(z_1,\ldots,z_n)}=\beta_{(z_{\sigma(1)},\ldots,z_{\sigma(n)})},$$
the automorphism $\Upsilon_{\Psi_\sigma}$  satisfies conditions (i) and (ii) of the lemma.
\end{itemize}
Suppose now that $\Phi$ and $\Psi$ are two matrices in $GL_n(\ZZ)$ and let
  $\Upsilon_\Phi$ and $\Upsilon_\Psi$ be automorphisms of  $C^*(H_n)$
  satisfying conditions (i) and (ii)  with respect to $\Phi$ and
  $\Psi$. Then one easily checks that $\Upsilon_\Phi\circ\Upsilon_\Psi$  satisfies 
  conditions (i) and (ii)  with respect to $\Phi\circ\Psi$.
Since every matrix $\Phi$ of $GL_n(\ZZ)$ can be writen as a product of
elementary matrices and permutation matrices, this completes the proof.
\end{proof}

\begin{cor}\label{cor-twist-heis}
Let $\Psi$ be a matrix in  $GL_n(\ZZ)$. Then the pull-back
$(\Lambda_2\Psi)^*C^*(H_n)\left(\TT^{\frac{n(n-1)}{2}}\right)$ of $C^*(H_n)\left(\TT^{\frac{n(n-1)}{2}}\right)$ with respect to the homeomorphism $\Lambda_2\Psi:\TT^{\frac{n(n-1)}{2}}\to \TT^{\frac{n(n-1)}{2}}$ and equipped with the $\TT^n$-action 
$\Lambda_2\Psi^*\beta^n$ is $\TT^n$-equivariantly isomorphic to {$C^*(H_n)_\Psi\left(\TT^{\frac{n(n-1)}{2}}\right)$, i.e.}
$C^*(H_n)\left(\TT^{\frac{n(n-1)}{2}}\right)$ equipped with the $\TT^n$-action $\beta^n\circ\Psi^{-1}$.
\end{cor}
\begin{proof}
Let $\Upsilon_\Phi$ be an automorphism of $C^*(H_n)$ as in Lemma
\ref{lem-aut-heis}. Then the  lemma implies that the morphism
$C^*(H_n)\left(\TT^{\frac{n(n-1)}{2}}\right)\mapsto\Lambda_2\Psi^*C^*(H_n)\left(\TT^{\frac{n(n-1)}{2}}\right);\,x\mapsto
1\otimes_{\Lambda_2\Psi} \Upsilon_\Phi(x)$ is  an isomorphism of C*-algebra bundles
 over $\TT^{\frac{n(n-1)}{2}}$ which is equivariant with respect to the
 prescribed actions.
\end{proof}

\begin{lem}\label{lem-action}
Let  $A(X)$ be a NCP $\TT^n$-bundle 
 and let $q:Y\to X$ be any principal $\TT^n$-bundle. Then, for any $\Psi\in GL_n(\ZZ)$, we have
 $$(Y*A)_{\Psi}(X)=Y_{\Psi}*A_{\Psi}(X).$$
\end{lem}
\begin{proof}
Recall that $Y*A(X)$ is the fixed-point algebra of $C_0(Y)\otimes_qA(X)$ under the diagonal
action of $\TT_n$ (taking the inverse action on the first factor). It is immediately clear that this 
algebra does not change if we compose both actions with any fixed automorphism of $\TT^n$.
Thus we see that the identity  map gives the desired identifications.
\end{proof}

Recall that Theorem \ref{thm-T-bundle} provides a classification of NCP $\TT^n$-bundles
by pairs   $([q:Y\to  X], f)$, where $q:Y\to X$ is a (commutative) principal $\TT^n$-bundle and 
$f:X\to \TT^{\frac{n(n-1)}{2}}$ is a continuous map. We are now able to describe the change 
in these data if we pass from $A(X)$ to $A_{\Psi}(X)$ for $\Psi\in GL_n(\ZZ)$.

\begin{thm}\label{thm-T-bundle-action}
Let  $A(X)$ be a NCP $\TT^n$-bundle over a  second countable locally compact space $X$ with 
classifying data  $([q:Y\to  X], f)$ as in Theorem \ref{thm-T-bundle}. 
 Then  $A_{\Psi}(X)$ is classified by the data 
 $([q_{\Psi}:Y_{\Psi}\to X], \Lambda_2\Psi\circ f)$ for all $\Psi\in GL_n(\ZZ)$.
 More precisely, if $A(X)=Y*(f^*(C^*(H_n))(X)$, then $A_{\Psi}(X)=Y_{\Psi}*((\Lambda_2\Psi\circ f)^*(C^*(H)))(X).$
 \end{thm}

\begin{proof} Since $A(X)$ is classified by the pair $([q:Y\to  X], f)$, it follows from 
Theorem \ref{thm-T-bundle} that $A(X)$ is $X\times \TT^n$-equivariantly Morita equivalent to 
$Y*(f^*C^*(H_n))(X)$. Composing all $\TT^n$-actions  by $\Psi^{-1}$  we 
observe that
\begin{align*}
f^*C^*(H_n)_{\Psi}(X)&=C_0(X)\otimes_f C^*(H_n)_{\Psi}\big(\TT^{\frac{n(n-1)}{2}}\big)\\
&\stackrel{(*)}{=}C_0(X)\otimes_f\big((\Lambda_2\Psi)^*C^*(H_n)\big)\big(\TT^{\frac{n(n-1)}{2}}\big)\\
&=f^*\big((\Lambda_2\Psi)^*C^*(H_n)\big)(X)\\
&=(\Lambda_2\Psi\circ f)^*C^*(H_n)(X),
\end{align*}
where $(*)$ follows from Corollary \ref{cor-twist-heis}. The proof then follows from 
Lemma \ref{lem-action}.
\end{proof}

\section{Local $\RKK$-triviality of NCP torus bundles}\label{sec-loctriv}

In this section we want to show, among other useful results, that all 
NCP torus bundles are locally $\RKK$-trivial. We refer to \cite{Kas1} for the 
definition of the $\RKK$-group $\RKK(X; A(X), B(X))$ for two C*-algebra bundles 
$A(X)$ and $B(X)$. We only remark here that the cycles are given by 
the usual cycles $(E,\phi,T)$ for Kasparov's bivariant $K$-theory group $\KK(A(X), B(X))$ 
with the single additional requirement that the left 
$C_0(X)$-actions on the Hilbert bimodule $E$ induced via the left 
action of $A(X)$ on $E$ coincides with the right $C_0(X)$-action on 
$E$ induced by the right action of $B(X)$ on $E$. We start with

\begin{prop}\label{prop-convex}
Let $\Delta$ be a contractible and compact space and let $A(\Delta\times X)$
be a  NCP $\TT^n$-bundle over $\Delta\times X$. Then for any element $z$ of
$\Delta$, the evaluation map 
$$e_z:A(\Delta\times X)\to
A(\{z\}\times X)$$
gives an invertible class 
$[e_z]\in \RKK(X;A(\Delta\times X),A(\{z\}\times X))$.
\end{prop}
Before proving this proposition, we state an immediate corollary.

\begin{cor}\label{cor-homotop}
Assume that two NCP $\TT^n$-bundles $A_0(X)$ and $A_1(X)$ 
are homotopic in the sense that 
there exists a NCP $\TT^n$-bundle $A([0,1]\times X)$ which restricts 
to $A_0(X)$ and $A_1(X)$ at $0$ and $1$, respectively.
Then $A_0(X)$ and $A_1(X)$ are $\RKK$-equivalent.
\end{cor}
In view of the classification results of the previous section,
we then get 

\begin{cor}\label{cor-homotopic}
Assume that $f,g:X\to \TT^\frac{n(n-1)}{2}$ are two homotopic continuous functions.
Then $f^*C^*(H_n)(X)$ and $g^*C^*(H_n)(X)$ are $\RKK$-equivalent.
\end{cor}
\begin{proof} Apply the previous corollary to $F^*(C^*(H_n))([0,1]\times X)$
where $F:[0,1]\times X\to \TT^\frac{n(n-1)}{2}$ is a homotopy between 
$f$ and $g$.
\end{proof}

The above corollary then easily implies

\begin{cor}\label{cor-localRKK-triv}
Suppose that $A(X)$ is a NCP $\TT^n$-bundle. Then $A(X)$ is locally 
$\RKK$-trivial, i.e., for every $x\in X$ there exists an open  neighbourhood 
$U$ of $X$ such that $A(U)$ is $\RKK(U,\cdot,\cdot)$-equivalent to
$C_0(U, A_x)$.
\end{cor}
\begin{proof} Let $x\in X$ be fixed. By the classification results of the previous 
section we may assume that $A(X)=Y*(f^*(C^*(H_n)))$. But restricting to a small 
neighbourhood $U$ of $x$ we may assume that $Y$ is the trivial $\TT^n$-bundle 
and that $f$ is homotopic to the constant map with value $f(x)$. 
The result then follows from the previous corollary.
\end{proof}

For the proof of the proposition
we need the following strong version of the Thom-Connes isomorphism for
crossed products by $\RR^n$, which, as we point out below, 
 is due to Kasparov:

\begin{thm}\label{thm-Thom}
Suppose that $A(X)$ is a C*-algebra bundle and that $\alpha:\RR^n\to\Aut(A(X))$
is a fibre-wise action on $A(X)$. Then there exists 
an invertible class $$\frak t\in \RKK_n(X; A(X), A(X)\rtimes \RR^n),$$ i.e., the 
algebra $A(X)$ is $\RKK$-equivalent to $A(X)\rtimes \RR^n$ up to a dimension shift 
of order $n$ (mod $2$).
\end{thm}
\begin{proof} 
We use the  Dirac element $D\in \KK^{\RR^n}_n(C_0(\RR^n),\CC)$ as constructed 
by Kasparov in \cite[\S 4]{Kas1}, where $\RR^n$ acts on itself by the translation action $\tau$
(a different, and probably easier description is given in \cite{Kas0}). 
Since $\RR^n$ is amenable,  the element $D$ is invertible 
with inverse given by the dual-Dirac element $\eta\in \KK^{\RR^n}_n(\CC, C_0(\RR^n))$
also constructed in \cite[\S 4]{Kas1}.
We then obtain an invertible element 
$$\sigma_{A(X)}(D)\in \RKK^{\RR^n}_n(X; A(X)\otimes C_0(\RR^n), A(X))$$
by tensoring $D$ with $A(X)$. 
Using the descent, we obtain an invertible element 
$$j_{\RR^n}(\sigma_{A(X)}(D))
\in \RKK^{\RR^n}_n(X; \big(A(X)\otimes C_0(\RR^n)\big)\rtimes \RR^n, A(X)\rtimes \RR^n).$$
But the crossed product  $\big(A(X)\otimes C_0(\RR^n)\big)\rtimes \RR^n$ is 
$C_0(X)$-linearly isomorphic to $$A(X)\otimes \big(C_0(\RR^n)\rtimes\RR^n\big)\cong A(X)\otimes
\K(L^2(\RR^n)),$$ where the first isomorphism follows by first identifying 
$A(X)\otimes C_0(\RR^n)$ with $C_0(\RR^n, A(X))$ and then observing that
$$\varphi: C_0(\RR^n, A(X))\to C_0(\RR^n, A(X));\big( \varphi(g)\big)(s)=\alpha_{-s}(g(s))$$
is a bundle isomorphism which 
 transforms the diagonal action $\alpha\otimes \tau$ to $\id\otimes \tau$
on $A(X)\otimes C_0(\RR^n)$. We then use the general fact that
$(A\otimes_{\max} B)\rtimes_{\id\otimes \beta} G\cong A\otimes_{\max}(B\rtimes _\beta G)$
together with the fact that
 $C_0(\RR^n)\rtimes_{\tau}\RR^n\cong \K(L^2(\RR^n))$.
 We thus obtain a $C_0(X)$-linear Morita equivalence
 $A(X)\sim_M\big(A(X)\otimes C_0(\RR^n)\big)\rtimes \RR^n$. Multiplying 
 $j_{\RR^n}(\sigma_{A(X)}(D))$ from the right with this Morita equivalence gives 
the desired  invertible element $\frak t\in \RKK(X; A(X); A(X)\rtimes\RR^n)$.
 \end{proof}

As a consequence of the above theorem, we now obtain the following lemma:

\begin{lem}\label{lem-nctorusfibration}
Suppose that $\alpha:\ZZ^n\to \Aut(C_0(X,\K))$ is any fibre-wise action.
Then $A(X):=C_0(X,\K)\rtimes \ZZ^n$ is, up to a dimension shift of order $n$, 
$\RKK(X;\cdot,\cdot)$-equivalent to a continuous trace algebra $B=B(X\times \TT^n)$ with base 
$X\times \TT^n$.
\end{lem}
\begin{proof}
Consider the induced 
algebra $B:=\Ind_{\ZZ^n}^{\RR^n}C_0(X,\K)$, which can be defined as the set of all
bounded continuous functions $F:\RR^n\times X\to \K$ such that
$$F(s+m, x)=\alpha^x_m(F(s,x))\quad \text{for all}\; s\in \RR^n, m\in \ZZ^n \;\text{and}\; x\in X.$$
It is shown in \cite{RW} that this is a continuous trace algebra with base 
$X\times (\RR^n/\ZZ^n)\cong X\times \TT^n$,
since $\ZZ^n$ acts trivially on $X$.
Moreover, it follows from Green's imprimitivity theorem (e.g. see \cite{Will} for a complete treatment) that
the crossed product $A(X)=C_0(X,\K)\rtimes \ZZ^n$ is
Morita equivalent over $X$ to $\big(\Ind_{\ZZ^n}^{\RR^n}C_0(X,\K)\big)\rtimes \RR^n$, where $\RR^n$ acts  on the induced algebra by translation in the first variable of $F$. 
Using the above strong version of the  Thom isomorphism, we see 
that  $A(X)=C_0(X,\K)\rtimes \ZZ^n$ is
$\RKK(X;\cdot,\cdot)$-equivalent to $B(X\times \TT^n):=\Ind_{\ZZ^n}^{\RR^n}C_0(X,\K)$ 
up to a dimension shift of order $n$.
\end{proof}

We are now ready for

\begin{proof}[Proof of Proposition \ref{prop-convex}]
Let $A(\Delta\times X)=C_0(\Delta\times X,\K)\rtimes_{\alpha}\ZZ^n$ as in 
Proposition \ref{prop-dualbundle}.
It follows then from  Lemma \ref{lem-nctorusfibration} that 
$A(\Delta\times X)=C_0(\Delta\times X,\K)\rtimes_{\alpha}\ZZ^n$ is $\RKK(\Delta\times X; 
\cdot,\cdot)$-equivalent to a continuous trace algebra $B(\Delta\times
X\times \TT^n)$ {with Dixmier-Douady class $\delta$ in $H^3(\Delta\times
X\times \TT^n,\ZZ)$. Since $\Delta$ is
compact and 
contractible,   $\delta=\pi^*(\delta')$, where $\delta'$ lies in
$H^3(X\times\TT^n,\ZZ)$ and $\pi^*$ is the isomorphism induced in
cohomology by the projection  $\pi:\Delta\times
X\times \TT^n\to
X\times\TT^n$.}
{Hence} {this bundle is 
isomorphic  to some algebra of the form $\pi^*D(X\times \TT^n)=C(\Delta, D(X\times \TT^n))$
as C*-algebra bundles over $\Delta\times X\times \TT^n$, where
$D(X\times \TT^n)$  a continuous trace algebra over $X\times \TT^n$.}
Since $\Delta$ is a contractible compact space, the evaluation 
$e_z$ at an element $z$ of $\Delta$ induces
an $\RKK(X;\cdot,\cdot)$-equivalence.
\end{proof}

\section{The $K$-theory group bundle of a NCP torus bundles. }\label{sec-bundle}

Suppose that $X$ is a locally compact space. By an (abelian)
 {\em group bundle} 
 $$\mathcal G:=\{G_x: x\in X\}$$
we understand a family of groups $G_x$, $x\in X$, together with 
group isomorphisms  {$c_{\gamma}:G_x\to G_y$} for each continuous path
$\gamma: [0,1]\to X$ which starts at $x$ and ends at $y$, such that 
 the following 
additional requirements are  satisfied:
\begin{enumerate}
\item If $\gamma$ and $\gamma'$ are homotopic paths from $x$ to $y$, then 
$c_\gamma=c_{\gamma'}$.
\item {If $\gamma_1:[0,1]\to X$ and $\gamma_2:[0,1]\to X$ are paths from $y$ to $z$ 
and from $x$ to $y$, respectively, then 
$$c_{\gamma_1\circ\gamma_2}=c_{\gamma_1}\circ c_{\gamma_2},$$
where $\gamma_1\circ \gamma_2:[0,1]\to X$} is the usual composition of paths.
\end{enumerate}
It follows from the above requirements, that if $X$ is path connected, then all
groups $G_x$ are isomorphic and that we get a canonical action of the fundamental group 
$\pi_1(X)$ on each fibre $G_x$.

A morphism between two group bundles 
$\mathcal G=\{G_x: x\in X\}$ and $\mathcal G' =\{G'_x: x\in X\}$ is a family of group homomorphisms
$\phi_x:G_x\to G'_x$ which commutes with the maps $c_\gamma$. The trivial group bundle
is the bundle with every $G_x$ equal to a fixed group $G$ and all maps $c_{\gamma}$ 
being the identity. We then write $X\times G$ for this bundle.
If $X$ is path connected, then a given group bundle $\mathcal G$ on $X$ 
can be trivialized if and only if the action of $\pi_1(X)$ on the fibres $G_x$ are trivial.
In that case every path $\gamma$ from a base point $x$ to a base point $y$ induces 
the same morphism
$c_{x,y}:G_x\to G_y$ and the family of maps $\{c_{x_0,x}:x\in X\}$ is a group bundle 
isomorphism between the trivial group bundle $X\times G_{x_0}$  and the given 
bundle $\mathcal G=\{G_x:x\in X\}$ if we fix a base point $x_0$.

Suppose now that $A(X)$ is an NCP $\TT^n$-bundle over $X$.
Recall from Proposition \ref{prop-convex} that for any compact contractible 
space $\Delta$ and any map $f:\Delta\to X$
 the evaluation map $\eps_t:f^*A(\Delta)\to
A_{f(t)}$   is a $\KK$-equivalence. This implies in particular, that $A(X)$ is 
$\KK$-fibration  in the sense of the following definition:

 \begin{defn}\label{def-K-fibration}
A C*-algebra bundle $A(X)$ is called a {\em $\KK$-fibration} (resp. {\em $K$-fibration})
if for every compact contractible space $\Delta$ and any $v\in \Delta$ the evaluation 
map $ev_v:f^*A(\Delta)\to A_{f(v)}$ is a $\KK$-equivalence (resp. induces an isomorphism 
of $K_*(f^*A(\Delta))\cong K_*(A_{f(v)})$).
\end{defn}

\begin{prop}\label{prop-bundle}
Suppose that $A(X)$ is  a $K$-fibration. For any path $\gamma:[0,1]\to X$ 
with starting point $x$ and endpoint $y$ let
{$c_{\gamma}:K_*(A_x)\to K_*(A_y)$} denote the 
composition 
\begin{equation}\label{eq-cgamma}
{\begin{CD}
K_*(A_x) @>\eps_{0,*}^{-1}>\cong> K_*(\gamma^*A) @>\eps_{1,*}>\cong> K_*(A_y)
\end{CD}}
\end{equation}
Then $\mathcal K_*(A):=\{K_*(A_x): x\in X\}$ together with the above defined 
maps $c_{\gamma}$ is a group bundle over $X$.
\end{prop}
\begin{proof}
It is clear that constant paths induce the identity maps
and that $c_{\gamma\cdot \gamma'}=c_{\gamma}\circ c_{\gamma'}$, where $\gamma\cdot \gamma'$
denotes composition of paths.
Moreover, if $\Gamma:[0,1]\times [0,1]\to X$ is a homotopy between
the paths $\gamma_0$ and $\gamma_1$ with equal starting and endpoints, then 
 $c_{\gamma_0}$ and $c_{\gamma_1}$ both coincide with the 
composition {$\eps_{(1,1),*}\circ \eps_{(0,0),*}^{-1}$}, where $\eps_{(0,0)}$ and $\eps_{(1,1)}$
denote evaluation of $\Gamma^*A$ at the respective corners of $[0,1]^2$.
Hence we see that $c_\gamma$ only depends on the homotopy class of $\gamma$.
\end{proof}

\begin{defn}\label{defn-group-bundle}
Suppose that  $A(X)$  is a $K$-fibration (e.g., if $A(X)$ is a NCP $\TT^n$-bundle). 
Then
$\mathcal K_*(A):=\{K_*(A_x): x\in X\}$   together with 
the maps $c_{\gamma}:K_*(A_y)\to K_*(A_x)$  is called  the
{\em $K$-theory group bundle} associated to 
$A(X)$. 
\end{defn}

\begin{remark}\label{rem-KK}
 If $A(X)$ and $B(X)$ are two $K$-fibrations, then a class $\frak x\in 
\RKK(X, A(X), B(X))$ determines a morphism $\frak x_*: \mathcal K_*(A)\to \mathcal K_*(B)$ 
 given fibrewise by taking right 
Kasparov products with the evaluations $\frak x(x)\in \KK(A_x, B_x)$. Of course, these maps are 
isomorphisms of group bundles if and only $\frak x(x)$ induces an isomorphism 
$K_*(A_x)\cong K_*(B_x)$  for all $x\in X$.
This is certainly true if all $\frak x(x)$ are invertible in $\KK(A_x, B_x)$.
\end{remark}

In what follows next, we want to show that the $K$-theory group bundle 
of a NCP $\TT^n$-bundle $A(X)$ does not change if we twist this bundle
with any principal $\TT^n$-bundle $Y\stackrel{q}{\to}X$, i.e., we have
$\K_*(A(X))=\K_*(Y*A(X))$ for all such $Y$ (see \S \ref{sec-torus} for the notation).

Recall that $Y*A(X)$ is the fixed-point algebra of $q^*A=C_0(Y)\otimes_{q}A$
with respect to the diagonal $\TT^n$-actions (taking the inverse action on the first factor).
Evaluating at the fibres, we see that an element $b\in q^*A(X)$ is in $Y*A(X)$ if and 
only if $b(zy)=z^{-1}\cdot b(y)$, where on the right-hand side we consider the action 
of $z^{-1}\in \TT^n$ on the fibre $q^*A_y=A_{q(y)}$.
The fibre of $q^*A(X)$ over $x\in X$ is then given by 
$C(Y_x,A_x)$ with $Y_x=q^{-1}(\{x\})\cong \TT^n$, and hence the fibre $(Y*A)_x$
of $Y*A(X)$ at $x\in X$ is {canonically isomorphic} to the {fixed point algebra of $C(Y_x,A_x)\cong C(\TT^n, A_x)$ under }
the action
$$(z\cdot \varphi)(y)=z^{-1}\cdot \varphi(zy)\quad\text{for all}\; z\in \TT^n, \varphi\in C(Y_x,A_x),$$
from which we easily see that evaluation at any point $y\in Y_x$ induces a $\TT^n$-equivariant
isomorphism 
$$(Y*A)_x\cong A_x\quad
\text{for all $x\in X$}.$$

More generally, if $x$ is given, we can choose
an open neighbourhood $U$ of $x$   which trivializes the principal
$\TT^n$-bundle $Y$. Let $\phi: U\times \TT^n\to q^{-1}(U)$ be a trivializing map.
Then $Y*A(U)$ is the fixed point algebra of
\begin{equation}\label{equ-id}\begin{array}{rcl}
q^*A(q^{-1}(U))&=&C_0(q^{-1}(U))\otimes_{C_0(U)}A(U)\\
&\stackrel{\phi}{\cong}& C_0(\TT^n\times U)\otimes_{C_0(U)}A(U)\\
&\cong& C(\TT^n, A(U)).
\end{array}\end{equation}
Composing this chain of isomorphisms with 
$$C(\TT^n, A(U))\to A(U);\,f\mapsto f(1),$$ 
we obtain  an isomorphism
$$\Psi_{\phi}: Y*A(U)\to A(U)$$
between
the fixed point algebra $Y*A(U)$ of $q^*A(q^{-1}(U))$ and $A(U)$.
Looking carefully at the constructions, one finds that
$\Psi_{\phi}$ sends 
a section $b\in Y*A(U)$, which is given by a $\TT^n$-invariant section in $q^*A(q^{-1}(U))$,
 to the section of $A(U)$ given by 
$x\mapsto b(\phi(1,x))$. At the point $x\in U$ this induces the isomorphism
$Y*A_x\cong A_x$ considered above given by evaluation at $y=\phi(x,1)\in Y_x$.

\begin{lem}\label{lem-pullback}
Let $X$ and $Z$   be  locally compact spaces, let  $A(X)$ be a C*-algebra  
bundle over $X$ equipped with
a fiber-wise action of $\TT^n$ and let $Y\stackrel{q}{\to}X$ be a  principal
$\TT^n$-bundle. For any continuous map $h:Z\to X$, we denote by
$h^*Y\stackrel{q_h}{\to} Z$ the principal $\TT^n$-bundle  pulled
back from $Y\stackrel{q}{\to}X$ by $h$.
Then $h^*(Y*A)$ and
$h^*Y*h^*A$ are {canonically} isomorphic C*-algebra bundles over $Z$ and 
the following diagram commutes
$$
\begin{CD}K_*(h^*(Y*A)_z) @>\cong>>K_*(({h}^*Y*h^*A)_z)\\
@V\lambda_{h(z)} VV      @V\lambda_z VV\\
K_*(A_{h(z)})@>=\!=>> K_*(A_{h(z)}),
\end{CD}
$$
where 
\begin{itemize}
\item the top isomorphism is induced by the
bundle-isomorphism $h^*(Y*A)\to{h}^*Y*h^*A$ at the fibre over $z$,
\item the left vertical arrow is given by the isomorphism $h^*(Y*A)_z= (Y*A)_{h(z)}\cong A_{h(z)}$,
\item the right vertical arrow  is given by the isomorphism 
$(h^*Y*h^*A)_z\cong (h^*A)_z\cong A_{h(z)}$,
\end{itemize}
with  $(Y*A)_{h(z)}\cong A_{h(z)}$ and $(h^*Y*h^*A)_z\cong (h^*A)_z$ 
as in the above discussion.
\end{lem}
\begin{proof}
The algebra $h^*Y*h^*A$ is the fixed point algebra of
$$\big(C_0(Z)\otimes_{C_0(X)}C_0(Y)\big)\otimes_{C_0(Z)}\big(C_0(Z)\otimes_{C_0(X)}A\big)$$
under the diagonal action of $\TT^n$ which is given by the $\TT^n$-action on
$Y$ for the first factor and the inverse of the given $\TT^n$-action on $A$ for the second.
On the other hand, $h^*(Y*A)=C_0(Z)\otimes_{C_0(X)}(Y*A)$ can be 
identified with the fixed point algebra of 
$$C_0(Z)\otimes_{C_0(X)}\big(C_0(Y)\otimes_{C_0(X)}A),$$
with $\TT^n$ acting on $Y$ and $A$ precisely as above. Thus, to see that 
the fixed point algebras are isomorphic, it suffices to find an isomorphism 
$$\big(C_0(Z)\otimes_{C_0(X)}C_0(Y)\big)\otimes_{C_0(Z)}\big(C_0(Z)\otimes_{C_0(X)}A\big)
\cong C_0(Z)\otimes_{C_0(X)}\big(C_0(Y)\otimes_{C_0(X)}A),$$
which respects this $\TT^n$-action. But it is straight-forward
 to check that such isomorphism is given on elementary tensors
by
$$(g_1\otimes_X f)\otimes_Z(g_2\otimes_Xa)\mapsto g_1g_2\otimes_X(f\otimes_Xa)$$
whenever $g_1,g_2\in C_0(Z), f\in C_0(Y)$ and $a\in A$.  
Let us denote this isomorphism by $\Psi$ and we denote the induced isomorphism
on the fixed-point algebras by 
$$\Psi^{\TT^n}:h^*Y*h^*A\to h^*(Y*A).$$
This map is certainly $C_0(Z)$-linear, and since on both algebras 
$h^*Y*h^*A$  and $h^*(Y*A)$ the $\TT^n$-action is induced by the given action on 
the factor $C_0(Y)$, it is also $\TT^n$-equivariant. This proves the 
first statement of the lemma. 

In order to prove commutativity of the $K$-theory diagram, we first note that 
for a given $z\in Z$ the fibres of $(h^*Y*h^*A)_z$ and $h^*(Y*A)_z$ are 
the fixed point algebras of the fibres over $z$ of the algebras
$$\big(C_0(Z)\otimes_{C_0(X)}C_0(Y)\big)\otimes_{C_0(Z)}\big(C_0(Z)\otimes_{C_0(X)}A\big)
\quad\text{and}\quad C_0(Z)\otimes_{C_0(X)}\big(C_0(Y)\otimes_{C_0(X)}A),$$
respectively. In both cases, this fibre is given by
$C_0(Y_{h(z)}, A_{h(z)}),\quad\text{with} \quad Y_{h(z)}=q^{-1}(\{h(z)\})\subseteq Y,$
where for the first algebra the quotient map sends an elementary tensor 
$(g_1\otimes_X f)\otimes_Z(g_2\otimes_Xa)$ to 
$g_1(z)g_2(z)(f|_{Y_{h(z)}}\otimes a(h(z)))$
and for the second algebra an elementary tensor $g\otimes_X(f\otimes_X a)$ is mapped to
$g(z)(f|_{Y_{h(z)}}\otimes a(h(z))).$
It  follows from this description that the isomorphism $\Psi$ constructed above
intertwines these evaluations at $z$, and hence the isomorphism 
$\Psi^{\TT^n}:h^*Y*h^*A\to h^*(Y*A)$ induces the identity on 
the fixed point algebra of  $C_0(Y_{h(z)},A_{h(z)})$ under these evaluations.
The homomorphisms in the vertical arrows of the diagram in the lemma 
now only depend on the choice of two possible elements $y_1, y_2\in Y_{h(z)}$
on which we evaluate an invariant function $f\in C(Y_{h(z)}, A_{h(z)})$.
Since $\TT^n$ acts transitively on $Y_{h(z)}$ we find
$u\in \TT^n$ such that $y_1=u\cdot y_2$. 
We then have 
$$\lambda_{h(z)}(f)=f(y_1)=f(u\cdot y_2)=u^{-1}\cdot (f(y_2))=u^{-1}(\lambda_z(f)).$$
this shows that on the level of algebras, the diagram in the Lemma commutes up 
to an automorphism of the fibre $A_{h(z)}$ given by the action of a fixed element 
$u\in \TT^n$. Since $\TT^n$ is connected, this automorphism is homotopic to the identity,
and therefore induces the identity map on $K_*(A_{h(z)})$.
\end{proof}

 \begin{remark}\label{rem-auto}
Note that we saw in the discussion preceding the above lemma, that the 
isomorphism $\lambda_x:(Y*A)_x\cong A_x$ depends on a choice of an element 
$y\in Y_x$. However, the last argument in the above proof shows that 
the $K$-theory map $\lambda_{x,*}:K_*((Y*A)_x)\cong K_*(A_x)$ 
is independent from this choice.
\end{remark}

We are now ready to show that the K-theory group bundle of a NCP $\TT^n$-bundle $A(X)$ 
is invariant under the action of a principal $\TT^n$-bundle $Y\stackrel{q}{\to}X$:

\begin{lem}\label{lem-group-twisted-bundle}
Let $A(X)$  be an NCP $\TT^n$-bundle and let $Y\stackrel{q}{\to}X$
be any principal $\TT^n$-bundle over $X$. Then 
 for any path $\mu:[0,1]\to X$ there is a $C[0,1]$-isomorphism
 $\Lambda:\mu^*(Y*A(X))\to\mu^*A(X)$ such that for any 
 $t\in[0,1]$, the following diagram commutes:
$$
\begin{CD}  K_*(\mu^*(Y*A)) @>\eps_{t,*} >> K_*((Y*A)_{\mu(t)})\\
@V\Lambda_* VV      @V\lambda_{\mu(t)} VV\\
K_*(\mu^*A)@>\eps_{t,*}>> K_*(A_{\mu(t)}),
\end{CD}
$$
where $\eps_t$ is the evaluation at $t$. As a consequence, the 
$K$-theory group bundles of $A(X)$ and of $Y*A(X)$ coincide.
\end{lem}
\begin{proof} According to Lemma \ref{lem-pullback}, the algebras
$\mu^*(Y*A)$ and ${\mu}^*Y*{\mu}^*A$ are $C[0,1]$-isomorphic.
 Since $[0,1]$ is contractible,  the
principal $\TT^n$-bundle  ${\mu}^*Y$ is trivializable. In view of
the discussion at the beginning of the subsection, the choice of a
trivialization induces a $C[0,1]$-isomorphism
$\Lambda:\mu^*(Y*A)\to\mu^*A$ and for $t$ in $[0,1]$, we have
$\Lambda_{t,*}=\lambda_{\mu(t)}$. The commutativity of the diagram is
now just the compatibity between $\Lambda$ and the evaluation maps.
\end{proof}

\section{The action of $\pi_1(X)$ in the two-dimensional case}\label{sec-dim2}
In this section we want to study the $K$-theory group bundle 
for NCP $\TT^2$-bundles. In this case the
representation group $H_2$ is the discrete Heisenberg group, and we 
have to study the group algebra $C^*(H_2)$ as a NCP $\TT^2$-bundle over $\TT$.

Recall that $C^*(H_2)$ is  the C*-algebra generated by the
unitaries $U,V$ and $W$ and the  relations  $$UW=WU,\,VW=WV \text{ and
  } UV=WVU. $$ 
The $C(\TT)$-structure is then provided
by the central unitary $W$ and the fiber at $z=e^{2\pi i\theta}\in\TT$, for
$\theta\in\RR$, is  the noncommutative torus $A_\theta$ generated by the
unitaries $U_\theta$ and $V_\theta$ and the relation
$U_\theta V_\theta=e^{2\pi i\theta}V_\theta V_\theta$. The $K$-theory
groups 
for noncommutative tori are 
$$K_0(A_\theta)\cong K_1(A_\theta)\cong\ZZ^2.$$
 In order to describe
the action of $\pi_1(\mathbb{T})\cong\ZZ$ on the $K$-theory group
bundle of  $C^*(H_2)$, we shall first recall the standard generators for $K_*(A_\theta)$:
\begin{itemize}
\item The standard generators for $K_1(A_\theta)$ are the classes of
  the unitaries $U_\theta$
  and $V_\theta$.
\item If $\theta\in(0,1)$, then the standard generators for
$K_0(A_\theta)$ are  the class  of the trivial projector $1\in A_\theta$
and the class of the Rieffel projector $p_\theta$.
\end{itemize}
The standard two generators for $K_0(C(\mathbb{T}^2))$  are 
\begin{itemize}
\item the
 class  of the trivial projector $1\in C(\mathbb{T}^2)$;
\item the Bott element $\beta$, i.e the unique element in $K_0(C(\mathbb{T}^2))$
with Chern character equal to the class in $H^2(\mathbb{T}^2,\RR)$  of the volume
form of the $2$-torus.
\end{itemize}

It follows from Proposition \ref{prop-bundle} together with the discussions preceding it,
that we obtain a natural action of $\pi_1(\TT)=\ZZ$ on $K_*(A_z)$ for all $z\in \TT$
and we are now going to compute this action.
 For $z=1$, the   action of the loop {generating $\pi_1(\TT)$}
$$\nu:[0,1]\to \TT;\,\theta \mapsto e^{2i\pi \theta}$$
on  $K_*(C(\TT^2))$ arising from the $K$-fibration  is given by 
$\epsilon_{1,*}\circ \epsilon_{0,*}^{-1}$ where
$$\epsilon_{j}:\nu^*(C^*(H_2))\longrightarrow
A_{j}\cong C(\TT^2)$$ is the evaluation at $j$ for $j=0,1$. 
Since the generators $U$ and $V$ of $C^*(H_2)$  provide global sections of unitaries  of the
C*-algebra bundle $C^*(H_2)(\TT)$ and since 
$U(1)$ and $V(1)$ are the standard generators of $K_1(C(\TT^2))$, we see that the
action of the loop $\nu$ on $K_1(C(\TT^2))$ is trivial. The unit  $1\in C^*(H_2)$ 
also provides a global section which 
evaluates at $1 \in C(\TT^2)$, from which we deduce that
  $[1]\in K_0(C(\TT^2))$ is also left invariant by the
  action of the loop $\nu$.
Let us now   compute the image under  {$\epsilon_{1,*}\circ
  \epsilon_{0,*}^{-1}$} of the generator $\beta$.
For this we define the morphism $$\iota_\nu:  C[0,1]\to
\nu^*(C^*(H_2))=C[0,1]\otimes_\nu C^*(H_2);\, f \mapsto
f\otimes_\nu 1$$ and we set
$$U'=1\otimes_\nu U\text{ and  }V'=1\otimes_\nu V$$ in $\nu^*(C^*(H_2))$. Let us 
define  a $\nu^*(C^*(H_2))$-valued inner product on $C_c([0,1]\times \RR)$ by 
$$\langle\xi,\eta\rangle=\sum_{(m,n)\in\ZZ^2}\iota_\nu(\langle\xi,\eta\rangle_{m,n})U'^mV'^n,$$
where for any pair of integers $(n,m)$,  $\langle\xi,\eta\rangle_{m,n}$ is the continuous function on
$C[0,1]$ defined by 
$$\langle\xi,\eta\rangle_{m,n}(\theta)=(\theta+1)\int_{-\infty}^{+\infty}{\bar{\xi}}(\theta,x+m\theta+m)\nu(\theta,x)e^{-2i\pi
  nx}dx\quad\theta\in [0,1].$$  
Evaluation of  $\langle\xi,\xi\rangle$ at each fiber
$\theta\in[0,1]$   is positive  in $A_{\theta}$ (see \cite{Rie}) and thus
$\langle\xi,\xi\rangle$ is a positive element of
$\nu^*(C^*(H_2))$. Let us denote by
$\mathcal{E}$ the completion of $C_c([0,1]\times \RR)$ with respect to the
above inner product.
There is a right action of $C[0,1]$ on $\mathcal{E}$ which is given on  $C_c([0,1]\times \RR)$  by
pointwise multiplication and a right action of $C^*(H_2)$ on
$\mathcal{E}$  given 
in the following way:
\begin{itemize}
\item $W$ acts by pointwise multiplication by $\nu$;
\item $\xi\cdot U(\theta,x)=\xi(\theta,x+\theta+1)$ for every  $\xi\in
  C_c([0,1]\times \RR)$;
\item $\xi\cdot V(\theta,x)=e^{2\pi i x}\xi(\theta,x)$ for every  $\xi\in
  C_c([0,1]\times \RR)$.
\end{itemize}
The action   of  $C[0,1]$ and of  $C^*(H_2)$ on $\mathcal{E}$  commute
and thus induce an action  of
$C[0,1]\otimes C^*(H_2)$. It is straightforward to check that this
action factorizes
through an action of $\nu^*(C^*(H_2))=C[0,1]\otimes_\nu C^*(H_2)$ on $\mathcal{E}$. We
define in this way a $\nu^*(C^*(H_2))$-right Hilbert module structure.
The algebra $\mathcal{K}(\mathcal{E})$ of compact operators on $\mathcal{E}$ is a
$C[0,1]$-algebra and according to \cite{Rie} the fiber of
$\mathcal{K}(\mathcal{E})$  at any
$\theta\in[0,1]$  is unital.

 The following well know lemma (e.g., see \cite[Proposition C.24]{Will}) applies to the
$C[0,1]$-algebra $\mathcal{K}(\mathcal{E})+\CC \Id_\mathcal{E}$ to
prove that $\mathcal{K}(\mathcal{E})$ is a unital $C^*$-algebra.
\begin{lem}\label{lem-isofiber}
Let $A(X)$ be a C*-algebra bundle over the compact space $X$ and let $B(X)\subseteq A(X)$ be
any closed  $C_0(X)$-invariant $*$-subalgebra of $A(X)$ 
with $A_x=B_x$ for every $x\in X$.
Then $A(X)=B(X)$.
\end{lem}

As a  consequence  the pair $(\mathcal \E, 0)$  defines a class in $\KK(\CC, \nu^*(C^*(H_2)))=
K_0(\nu^*(C^*(H_2)))$ which we
shall denote by $[\mathcal{E}]$.
In order to express $$\epsilon_{0,*}([\mathcal{E}])\in
K_0(C(\TT^2))\text{ and }\epsilon_{1,*}([\mathcal{E}])\in K_0(C(\TT^2))$$ in the basis of
standard generators, we need to compute for both of these elements
\begin{itemize}
\item their  rank;
\item the image of  the   Chern
character  under the fundamental class  in homology of $\TT^2$. 
\end{itemize}
Both computations have been carried out  in
\cite[p. 232]{Connes}\footnote{{Actually in \cite{Connes} the second computation was carried out
    in cyclic cohomology. The image of the Chern character was under
    the higher trace (i.e a pairing with a  cycle)
     corresponding to the fundamental
    class of $\TT^2$ under the Hochschild-Kostant-Rosenberg isomorphism
    between  de Rham homology of $\TT^2$ and the cyclic cohomology of
     $C^\infty(\TT^2)$}.}
from where we deduce { $$\epsilon_{0,*}([\mathcal{E}])=[1]+\beta\text{
  and }
\epsilon_{1,*}([\mathcal{E}])=2[1]+\beta.$$} Eventually, the action of 
the loop $\nu$ on $\in K_0(C(\TT^2))$ is given on the basis of standard
generators
$([1],\beta)$ by the  matrix $\begin{pmatrix}
  1& {1}\\0&1\end{pmatrix}$.
{For any irrational $\theta\in(0,1)$, 
Rieffel computes in \cite[Theorem 1.4]{Rie} that $\tau(\eps_{\theta,*}([\mathcal{E}])=\theta+1$,
and hence we get from the injectivity of the trace on $K_0(A_{\theta})$ that
$$\eps_{\theta,*}([\mathcal{E}])=[p_\theta]+[1],$$}
{where $p_\theta\in A_{\theta}$ denotes the Rieffel-projection.  
Thus, the
action of the (oriented)  generator of $\pi_1(\TT)\cong\ZZ$ on
$K_0(A_\theta)$ is given on the basis  of generators $([1],[p_\theta])$
by the  matrix $\begin{pmatrix} 1&1\\0&1\end{pmatrix}$.}
In what
follows, we will  denote 
by $M_z$ the endomorphism of $K_0(C^*(H_2)_z)$ corresponding to the
action of the  oriented generator of $\pi_1(\TT)$.

Assume now that $f:X\to \TT$ is any continuous map from the path connected 
locally compact space $X$ to $\TT$ and let $x_0\in X$ be any 
chosen base-point of
$X$. Next we want to compute the 
$K$-theory group bundle of the pull-back $f^*C^*(H_2)(X)$, which amounts 
in computing the action of $\pi_1(X)$ on the fibre $K_*(A_{f(x_0)})$.
Note that $f$ induces a map
$$\pi_1(X)\to \ZZ,\, \gamma \mapsto \langle f, \gamma\rangle,$$  where for any continuous map
$h:\TT\to X$ representing $\gamma$, the integer $\langle f, \gamma\rangle$
 is the winding number of $f\circ h:\TT\to\TT$.

\begin{prop}\label{prop-pi1}
Let $X$ be a path connected locally compact space and let $f:X\to\TT$ be a continuous
map.
Then for any fixed $x_0$ in $X$, the action of $\gamma\in\pi_1(X)$ on
$K_*(C^*(H_2)_{f(x_0)})$ arising from  the NCP $\TT^2$-bundle  $f^*(C^*(H_2))$ is trivial
on $K_1(C^*(H_2)_{f(x_0)})$ and it is given by $M_{f(x_0)}^{\langle
  f,\gamma\rangle}$ on $K_0(C^*(H_2)_{f(x_0)})$. 
\end{prop}
\begin{proof}
Let us represent $\gamma\in\pi_1(X)$ by a continuous map $g:[0,1]\to
X$ such that $g(0)=g(1)=x_0$. Then the action  of $\gamma$ on $K_*(C^*(H_2)_{f(x_0)})$
is given by 
$\epsilon_{0,*}\circ \epsilon_{1,*}^{-1}$ where
$$\epsilon_{j}:g^*(f^*(C^*(H_2))\longrightarrow C^*(H_2)_{f(x_0)}$$
 is the evaluation at $j=0,1$ of the $C[0,1]$-algebra $$g^*(f^*(C^*(H_2))\cong (f\circ
g)^*(C^*(H_2)).$$
The map $f\circ
g$, viewed as a map on the circle, has winding number $n={\langle
  f,\gamma\rangle}$. Hence it follows from the above computations for the 
  single loop of $\TT$  that the action of the loop  $f\circ
g$ on $K_*(C^*(H_2)_{f(g(0))})=K_*(C^*(H_2)_{f(x_0)})$ is  $M_{f(x_0)}^n$.
\end{proof}

\section{The action of $\pi_1(X)$ in the general case}\label{sec-general}
In this section we want to extend the results of the previous section 
to NCP $\TT^n$-bundles for arbitrary $n\in \NN$. We do this by 
exploiting the results for $\TT^2$-bundles obtained above.
The key for this is the study of the 
``universal'' NCP $\TT^n$-bundles $C^*(H_n)$ over $\widehat{Z}_n\cong \TT^{\frac{n(n-1)}{2}}$, where $$1\to Z_n\to H_n\to \ZZ^n\to 1$$ 
is the representation group of $\ZZ^n$ as described in \S \ref{sec-torus}.
To be  more precise, as we shall see below, the restriction of $C^*(H_n)$ to
certain ``basic circles'' gives a direct link between the $\TT^n$ case and the $\TT^2$-case.

Consider the circles  $\TT_{ij}$
which are given by the characters of the $i,j$-th coordinate of $\ZZ^{\frac{n(n-1)}{2}}$ for $i<j$.
These  are
represented by those matrices $\Theta$ which are everywhere zero except in 
the entry $\Theta_{ij}$. We call $\TT_{ij}, i<j$, a {\em basic circle} in 
$\TT^{\frac{n(n-1)}{2}}$. 
(If we identify $\ZZ^{\frac{n(n-1)}{2}}$
with the integer valued strictly upper trianguar matrices).
In what follows below we shall need a careful description of the restriction $C^*(H_n)|_{\TT_{ij}}$
of the universal NCP $\TT^n$-bundle $C^*(H_n)$ to the basic circles
$\TT_{ij}$. In order to do this, we should first state some  basic facts:
first of all, if $A(X)$ is any C*-algebra bundle with fibres $A_x$, then the space
$\widehat{A(X)}$ of unitary equivalence classes of irreducible representations of $A(X)$
is the disjoint union $\cup_{x\in X}\widehat{A}_x$, where the representations of $A_x$
are regarded as representations of $A(X)$ via composition with the quotient map
$\eps_x:A(X)\to A_x$ (e.g., see \cite[Theorem 10.4.3]{Dix}). 
If we regard $C^*(H_n)$ as a bundle over $\widehat{Z}_n$ with bundle structure
given via convolution with
$C_0(\widehat{Z}_n)\cong  C^*(Z_n)$, then we can identify 
$$\widehat{C^*(H_n)}_\chi=\{\pi\in \widehat{C^*(H_n)}: \pi|_{Z_n}=\chi\cdot 1_\pi\},$$
where we recall that by Schur's lemma every irreducible representation of $H_n$ 
restricts {on $Z_n$} to a multiple of a character $\chi\in Z_n=Z(H_n)$.
Using these facts, we get

\begin{lem}\label{lem}
The restriction $C^*(H_n)|_{\TT_{ij}}$ of $C^*(H_n)$ to the basic circle $\TT_{ij}$ 
is canonically isomorphic to $C^*(H_2)\otimes C(\TT^{n-2})$, as 
a non-commutative $\TT^n$-bundle over $\TT\cong \TT_{ij}$, where the 
$\TT^n$-action on $C^*(H_2)\otimes C(\TT^{n-2})$ is given via the action of the 
$i$th and $j$th coordinates as the $1$st and $2$nd coordinates of 
the non-commutative $\TT^2$-bundle $C^*(H_2)$ 
and the other coordinates acting (by their nummerical order) on the coordinates 
of the factor $C(\TT^{n-2})$.
\end{lem}
\begin{proof}
We actually prove that $C^*(H_n)|_{\TT_{ij}}$ is isomorphic to the 
$C^*$-group algebra of the group $H_n/Z_{ij}$, where 
$Z_{ij}:=\{M\in Z_n: M_{ij}=0\}$. Indeed, we have a canonical
quotient map 
$q: C^*(H_n)\to C^*(H_n/Z_{ij})$ and the kernel of this quotient map
is the intersection of all $C^*$-kernels of  the irreducible representations 
of $H_n/Z_{ij}$, viewed as representations of $H_n$.
But an arbitrary irreducible representation of $H_n$ restricts to 
the trivial representation of $Z_{ij}$ if and only if it restricts to a character 
in the basic circle $\TT_{ij}$ as described above. Thus it follows from the 
above mentioned basic facts  that $C^*(H_n)|_{\TT_{ij}}$ is isomorphic 
to $C^*(H_n/Z_{ij})$, and this isomorphism is clearly a non-commutative $\TT^n$-bundle
isomorphism if we equip $C^*(H_n/Z_{ij})$ with the $\TT^n$-bundle structure 
coming from the central extension 
$$ 1\to (\ZZ\cong Z_n/Z_{ij})\to H_n/Z_{ij}\to \ZZ^n\to 1.$$
The result then follows from the fact that the above extension is isomorphic to the extension
$$1\to \ZZ\to H_2\times\ZZ^{n-2}\to \ZZ^n\to 1$$
if we permute the $i$th and $j$th coordinate of $\ZZ^n$ in the first extension
to the first and second coordinate of $H_2$ in the second extension.
\end{proof}

We are now going to compute the action of 
$\pi_1\left(\TT^\frac{n(n-1)}{2}\right)\cong\ZZ^\frac{n(n-1)}{2}$ on
$K_*(C(\TT^n))$ arising from the
$K$-fibration $C^*(H_n)$  (note that
$C(\TT^n)$ is the fiber at $(1,\ldots,1)\in\TT^\frac{n(n-1)}{2}$ of
    $C^*(H_n)$).
    For this purpose, we will need an explicit  description of
    $K_0(C(\TT^n))$. We define for $i=1,\ldots,n$ the continuous
    functions 
$$u_i:\TT^n\to\CC;\, (z_1,\ldots,z_n)\mapsto z_i.$$ Let $(e_1,\ldots,e_n)$ be
    the canonical basis of $\ZZ^n$ and let $\Lambda_*(\ZZ^n)$ be the
    exterior algebra of $\ZZ^n$. 
\begin{lem}\label{lem-ext}
There is an isomorphism  of unital algebra  $\Lambda_*(\ZZ^n)\to K_*(C(\TT^n))$ uniquelly
defined on generators of $\ZZ^n$ by $e_i\mapsto [u_i]$.
\end{lem}
\begin{proof}
Since the  product $$K_1(C(\TT^n))\times K_1(C(\TT^n))\to
K_0(C(\TT^n))$$ is skew-symetric and by the universal property  of the
exterior algebra,  the  morphism of abelian groups
 $$\ZZ^n\to K_1(C(\TT^n));\, e_i\mapsto [u_i]$$
 extends  in  a unique way to a morphism of unital $\ZZ_2$-graded algebras
$$\Lambda_*(\ZZ^n)\to K_*(C(\TT^n)).$$ 
This is clearly an isomorphism  for $n=1$. Since we have 
isomorphisms of  $\ZZ_2$-graded algebras $$\Lambda_*(\ZZ^n)\cong
\Lambda_*(\ZZ^{n-1})\otimes \Lambda_*(\ZZ^2)$$ and $$K_*(C(\TT^n))\cong
K_*(C(\TT^{n-1})\otimes K_*(C(\TT)),$$ the  result follows 
for all other positive integers $n$  by induction.\end{proof}

In what follows below we write $z_{k,l}$ for the $k,l$-th component of $\TT^\frac{n(n-1)}{2}$
for all $1\leq k<l\leq n$. 
Then $\pi_1\left(\TT^\frac{n(n-1)}{2}\right)\cong\ZZ^\frac{n(n-1)}{2}$
has generators $\{\gamma_{i,j}:=[\nu_{i,j}]: 1\leq i<j\leq n\}$ with 
$$\nu_{i,j}:[0,1]\to \TT^\frac{n(n-1)}{2}\quad (\nu_{i,j})_{(k,l)}(\theta)=
\left\{
\begin{array}{cc}
e^{2\pi i\theta}&\text{ if }(k,l)=(i,j)\\
1&\text{ else}
\end{array}\right\}.$$
Let $M_{i,j}$ denote the automorphism of the fibre $\Lambda_*(\ZZ^n)\cong K_*(C(\TT^n))$
on the $K$-theory bundle $\K_*(C^*(H_n))$
 at $(1,\ldots,1)\in \TT^\frac{n(n-1)}{2}$ 
given by the action of the generator $\gamma_{i,j}=[\nu_{i,j}]\in \pi_1(\TT^\frac{n(n-1}{2})$. 
It is  
 given by 
$\epsilon_{1,*}\circ \epsilon_{0,*}^{-1}$ where
$$\epsilon_{l}:\gamma^*_{i,j}(C^*(H_n))\longrightarrow
C^*(H_n)_{(1,\ldots,1)}\cong C(\TT^n)$$
 is the evaluation at $l=0,1$ of the $C[0,1]$-algebra
 $\nu^*_{i,j}(C^*(H_n))$.
Since we have the obvious  idenfication $$\nu^*_{i,j}(C^*(H_n))\cong
\nu^*_{i,j}(C^*(H_n)|_{\TT_{i,j}}),$$ 
where on the right hand side we identify $\nu_{i,j}$ with the standard positively oriented 
paramatrization of the basic circle $\TT_{i.j}$,
we obtain from  Lemma \ref{lem} a canonical isomorphism of $C[0,1]$-algebras
$$\nu^*_{i,j}(C^*(H_n))\cong \nu^*(C^*(H_2))\otimes C(\TT^{n-2}))$$ with
$$\nu:[0,1] \to  \TT;\, \theta\mapsto e^{2i\pi \theta}.$$
The automorphism $M_{i,j}$ is therefore induced by the action of the
positive generator of $\pi_1(\TT)\cong\ZZ$ on $C^*(H_2)$, and it is then straightforward 
to check:

\begin{prop}\label{prop-action-of-pi1}
For any ordered  subset $J=\{i_1<\cdots<i_k\}$ of $\{1,\ldots,n\}$,
we set $$e_J=e_{i_1}\wedge\cdots \wedge e_{i_k}\in
\Lambda_*(\ZZ^n),$$ and for $i,j\in J$ such that $i<j$ we denote by $m_{J,i,j}$ the
 number of places between $i$ and $j$ in $J$. Then 
$$M_{i,j}\cdot
e_J=\left\{\begin{matrix} e_J{+}(-1)^{m_{J,i,j}}e_J{J\setminus\{i,j\}}, & \text{if $\{i,j\}\subset J$}\\
e_J & \text{else}\end{matrix}\right\}.$$
\end{prop}
Notice that since $\pi_1\left(\TT^\frac{n(n-1)}{2}\right)$ is
  commutative, the matrices $M_{i,j}$ commute with each other!
\begin{cor}\label{cor-winding}
The representation  $$M:\pi_1\left(\TT^\frac{n(n-1)}{2}\right)\longrightarrow
{GL}(K_*(C(\TT^n)))\cong {GL}(\Lambda_*(\ZZ^n)) $$ arising from the 
NCP $\TT^n$-bundle $C^*(H_n)\big(\TT^\frac{n(n-1)}{2}\big)$ 
is injective.
\end{cor}
\begin{proof}
Let $(n_{i,j})_{1\leq i<j\leq n}$ be an element of
$\pi_1\left(\TT^\frac{n(n-1)}{2}\right)\cong \ZZ^\frac{n(n-1)}{2}$. Then the
action of $(n_{i,j})_{1\leq i<j\leq n}$ on $K_*(C(\TT^n))$ is given 
by $$M\big((n_{i,j})_{1\leq i<j\leq n}\big)=\prod_{1\leq i<j\leq n}
M_{i,j}^{n_{i,j}}.$$ 
We have  $$M\big((n_{i,j})_{1\leq i<j\leq n}\big)\cdot e_k\wedge
  e_l= e_k\wedge
  e_l{+}n_{k,l}$$  for any
  $(k,l)$ with $1\leq k<l\leq n$. Therefore, $M\big((n_{i,j})_{1\leq i<j\leq n}\big)$ 
  is the identity map if and only if
  $n_{i,j}=0$ for every $1\leq i<j\leq n$.
\end{proof}
Let us now denote for $z\in \TT^\frac{n(n-1)}{2}$ by   $M_{i,j,z}$ the
action of the oriented  generator of
$\pi_1\left(\TT^\frac{n(n-1)}{2}\right)\cong \ZZ^\frac{n(n-1)}{2}$ on $K_*((C^*(H_n)_z))$
corresponding to $1\leq i<j\leq n$ and by
$$M_z:\pi_1\left(\TT^\frac{n(n-1)}{2}\right)\to{GL}(K_*(C^*(H_n)_z))$$ the
action of $\pi_1\left(\TT^\frac{n(n-1)}{2}\right)$ on $K_*(C^*(H_n)_z)$.
Since $M_z$ and $M$ are conjugate, we see that $M_z$ is also injective
for any $z\in \TT^\frac{n(n-1)}{2}$. 
\vspace{0.5cm}

We now describe for a
continuous map $f:X\to \TT^\frac{n(n-1)}{2}$, with $X$ a path connected locally compact space,
the action
$$M^f_x:\pi_1(X)\to \operatorname{GL}(K_*(C^*(H_n)_{f(x)}))$$
 arising for any $x\in X$  from the
$K$-fibration $f^*(C^*(H_n))$. We proceed as we did in the proof of
Proposition \ref{prop-pi1}.  For $\gamma\in\pi_1(X)$ represented by a
continuous map $g:\TT\to X$, the loop  $f\circ g:\TT\to
\TT^\frac{n(n-1)}{2}$ is homotopic to $\left(\nu^{\langle
  f_{i,j},\gamma\rangle}_{i,j}\right)_{1\leq i<j\leq n}$, where 
  $$f(x)=(f_{i,j}(x))_{1\leq i<j\leq n}\in\TT^\frac{n(n-1)}{2}$$
for any $x\in X$ (and the parametrisation  $\nu_{ij}$ of the basic circle
$\TT_{ij}$ is viewed as a function $\nu_{ij}:\TT\to \TT^\frac{n(n-1)}{2}$). Thus we obtain
\begin{equation}\label{eq-action}
M^f_x(\gamma)=\prod_{1\leq i<j\leq n} M_{i,j,f(x)}^{\langle
  f_{i,j},\gamma\rangle}.
  \end{equation}
    Since for any continuous functions
$h_j:X\to\TT,\,j=1,2$ and any $\gamma\in\pi_1(X)$, we have $${\langle
  h_{1}h_{2},\gamma\rangle}=\langle
  h_{1},\gamma\rangle+\langle
  h_{2},\gamma\rangle,$$ we get
\begin{lem}\label{lem-prod}
Let $f_1, f_2:X\to
\TT^\frac{n(n-1)}{2}$ be continuous functions such that
$f_1(x)=f_2(x)=(1,\ldots,1)\in \TT^\frac{n(n-1)}{2}$ 
for some fixed base point $x\in X$. Then,
for any $\gamma\in\pi_1(X)$, we have $$M^{f_1\cdot f_2}_{x}(\gamma)=M^{f_1}_{x}(\gamma)\circ
M^{f_2}_{x}(\gamma)$$ in  ${GL}(K_*(C(\TT^n)))$.
\end{lem}

We close this section by an explicit description of the action of $\pi_1(X)$ in the 
case $n=3$ if the fibre at $x$ is $C(\TT^3)$:

\begin{lem}\label{lem-dim3}
Let $A(X)$ be an NCP $\TT^3$-bundle over the path connected space $X$
with classification data $([q:Y\to X], f)$ 
such that $f(x)=1$ for some fixed $x\in X$. Choose 
$\{1, e_1\wedge e_2, e_2\wedge e_3, e_1\wedge e_3\}$  as a basis for the even part and 
$\{e_1, e_2, e_3, e_1\wedge e_2\wedge e_3\}$ as a basis for the odd part of
$K_*(C(\TT^3))=\Lambda_*(\ZZ^3)$.
Then the action of $\gamma\in \pi_1(X)$ on the even and odd parts of $K_*(C(\TT))=\Lambda_*(\ZZ^3)$ is  given with respect to those bases by the matrices {
$$\left(\begin{matrix} 1 &  \langle f_{12},\gamma\rangle &\langle f_{23},\gamma\rangle & \langle f_{13},\gamma\rangle\\
0&1&0&0\\
0&0&1&0\\
0&0&0&1\end{matrix}\right)\quad\text{and}\quad
\left(\begin{matrix} 1 &  0 & 0& \langle f_{23},\gamma\rangle\\
0&1&0&-\langle f_{13},\gamma\rangle\\
0&0&1&\langle f_{12},\gamma\rangle \\
0&0&0&1\end{matrix}\right),$$}
respectively.
\end{lem}
\begin{proof} This follows from a straightforward computation using 
Proposition \ref{prop-action-of-pi1}.
\end{proof}

\section{$\RKK$-equivalence of NCP torus bundles and the K-theory bundle}\label{RKK}

In this section we want to study $\RKK$-equivalence of NCP $\TT^n$-bundles 
$A(X)$ and $B(X)$ in terms 
of the K-theory bundles $\K_*(A(X))$ and $\K_*(B(X))$. Of course, one cannot expect a 
complete result because for this one would expect to need a much finer invariant.
However, as it turns out below, some partial answers can be given. In particular, we can see from 
the K-theory bundle whether $A(X)$ is $\RKK$-equivalent to a commutative principal bundle 
over $X$. We start with

\begin{prop}\label{prop-inj}
Let $X$ be a path connected locally compact space, and
let $f:X\to \TT^\frac{n(n-1)}{2}$ be a continuous map. Then, for all $x\in X$,  the action
$$M^f_x:\pi_1(X)\to{GL}(K_*(C^*(H_n)_{f(x)}))$$
 arising   from the NCP $\TT^n$-bundle  $f^*(C^*(H_n))$ is trivial if and only if $f$ is
homotopic to a constant map.
\end{prop}
\begin{proof} We denote by $[X,\TT]$ the set of homotopy classes of
  continuous functions from $X$ to $\TT$. The pointwise product of
  functions provides an abelian group structure on  $[X,\TT]$. If  $h:X\to\TT$ is a
  continuous functions, we denote by $[h]$ its homotopy class. Recall
  that if
  $[\TT]$ is the oriented generator of $H^1(\TT,\ZZ)\cong\ZZ$, then
  the map $$[h]\mapsto h^*([\ZZ])$$ induces  a group  isomorphism 
$$[X,\TT]\cong H^1(\TT,\ZZ)$$ (because $\TT$ is a $K(\ZZ,1)$-space).
If $\gamma$ is an element of $\pi_1(X)$, then we denote
  by $[\gamma]$ its class in $$H_1(X,\ZZ)\cong
  \pi_1(X)/[\pi_1(X),\pi_1(X)].$$ Under these notations, the number
$\langle h,\gamma\rangle$ is given by the pairing
$\langle [h],[\gamma]\rangle$ between a homology and a cohomology
class.   If $f:X\to \TT^\frac{n(n-1)}{2}$ is homotopic to the constant $c = f(x)$,
then $\langle [f_{i,j}],[\gamma]\rangle=0$  for any $1\leq
i<j\leq n$ and thus
$M^f_x(\gamma)$ is the identity map  on
$K_*(C^*(H^n)_{f(x)})$. Conversally,
if $M^f_x$ is the trivial morphism, this means that 
$\langle [f_{i,j}],[\gamma]\rangle=0$ for any $\gamma\in\pi_1(X)$ and
thus $\langle [f_{i,j}],\omega\rangle=0$ for any $\omega\in
H_1(X,\ZZ)$. Since $$H^1(X,\ZZ)\to  H_1(X,\ZZ)^*;\, \alpha\mapsto\langle
\alpha,\bullet\rangle$$ is an isomorphism, the map $f_{i,j}$ is
  homotopic to a constant map for any  $1\leq
i<j\leq n$ and therefore so is $f$.
\end{proof}

The following theorem is now a direct consequence of the above proposition
together with {Remark \ref{rem-KK}} and Lemma \ref{lem-group-twisted-bundle} :

\begin{thm} Let  $A(X)$ 
  be an NCP $\TT^n$-bundles
 over  the path connected space $X$ and let $f:X\to
  \TT^{\frac{n(n-1)}{2}}$  be the continuous map
  associated to $A(X)$  as
in Theorem \ref{thm-T-bundle}. Then the following are equivalent:
\begin{enumerate}
\item $f$ is homotopic to a constant map.
\item {The K-theory 
bundle  $\K_*(A(X))$ is trivial.}
\item $A(X)$ is $\RKK(X;\cdot,\cdot)$-equivalent to $C_0(Y)$ for 
a (commutative) principal $\TT^n$-bundle $q:Y\to X$.
\end{enumerate}
\end{thm}

In what follows next, we want to study whether the above result can be extended
in order to compare two NCP $\TT^n$-bundles $A(X)$ and $B(X)$ up to
twisting by suitable commutative $\TT^n$-bundles. Recall from Section 2
that if $A(X)$ is a NCP $\TT^n$-bundle and if $\Psi\in GL_n(\ZZ)$, then 
$A_{\Psi}(X)$ denotes the NCP $\TT^n$-bundle which we obtain from $A(X)$ 
by composing the $\TT^n$-action with $\Psi^{-1}$. Since the underlying C*-algebra 
bundle is not changed at all, and since our study of $\RKK$-equivalence is 
{\bf after} forgetting the $\TT^n$-action, we see that $A(X)$ and$A_{\Psi}(X)$
are $\RKK$-equivalent since they are equal after forgetting the action.

In terms of the classification of NCP $\TT^n$-bundles of Section 2, we 
can therefore deduce from Theorem \ref{thm-T-bundle-action} that the 
bundles corresponding to the pairs
$$([q:Y\to X], f)\quad\text{and}\quad ([q_{\Psi}:Y_{\Psi}\to X], \Lambda_2\Psi\circ f)$$
are always $\RKK$-equivalent (in fact they are $C_0(X)$-Morita equivalent).
More precisely, the representatives
$$Y*(f^*C^*(H_n))(X)\quad \text{and}\quad Y_{\Psi}*((\Lambda_2\Psi\circ f)^*(C^*(H_n)))(X)$$
are isomorphic C*-algebra bundles over $X$. 
We notice the following fact

\begin{lem}\label{lem-image}
The morphism $GL_n(\ZZ^n)\to
GL_n(\Lambda_2(\ZZ^n));\,\Psi\mapsto\Lambda_2\Psi$ is surjective for
$n=2$ and has range $SL(\Lambda_2(\ZZ^3))$ for $n=3$.
\end{lem}
\begin{proof}
For $n=2$, we have an isomorphism
$\Lambda_2(\ZZ^2)\stackrel{\cong}{\longrightarrow}\ZZ;\,x\wedge y\mapsto
\det(x,y)$ and under this identification $\Lambda_2\Psi$ is
$\det\Psi$.

For $n=3$  we have an isomorphism
$\Lambda_2(\ZZ^3)\stackrel{\cong}{\longrightarrow}(\ZZ^3)^*;\,x\wedge y\mapsto [z\mapsto
\det(x,y,z)]$ and under this identification $\Lambda_2\Psi$ is
$\det\Psi\cdot\, ^t\Psi^{-1}$, thus $\Psi\mapsto\Lambda_2\Psi$ has range
$SL(\Lambda_2(\ZZ^3))$.
\end{proof}

On the other hand, we know that 
$\RKK$-equivalence between $A(X)$ and $B(X)$  induces an isomorphism of the K-theory 
bundles $\K_*(A(X))\cong \K_*(B(X))$, which then implies in particular that 
the actions of $\pi_1(X)$ on the fibres $K_*(A_x)$ and $K_*(B_x)$ are conjugated.
If we can arrange things in a way that  $A_x=B_x=C(\TT^2)$ for some fixed $x\in X$,
this conjugation is given by some matrix $T\in GL(\Lambda_n(\ZZ))$.
Recall from Proposition \ref{prop-pi1} that in case $n=2$ and $f(x)=1\in \TT$
the action of $\gamma\in \pi_1(X)$
on $K_0(C(\TT^2))$ is given by the matrix 
$\left(\begin{smallmatrix} 1 & -\langle f,\gamma\rangle\\0& 1\end{smallmatrix}\right)$.
In case $n=3$ we get a similar description for the action on $K_0(C(\TT^3))$ by matrices 
of the form $\left(\begin{matrix} 1& ^tv\\ 0& I_3\end{matrix}\right)$, if we choose a basis of $K_0(C(\TT^3))$ as in Lemma \ref{lem-dim3}. Thus the following lemma
gives a description of the possible conjugation matrices in these cases:

\begin{lem}\label{lem-conjugate}
Let $v$ and $w$ be vectors in $\ZZ^{n}$ and let $T$ be a matrix of
$GL_{n+1}(\ZZ)$ such that
\begin{equation}\label{equ-T}T\cdot\begin{pmatrix}1&^tv\\0&I_{n}\end{pmatrix}=\begin{pmatrix}1&^tw\\0&I_{n}\end{pmatrix}\cdot
T.\end{equation} Then there is a matrix $X$ of $GL_{n}(\ZZ)$, depending only on
$T$ such that $v=X\cdot w$.
\end{lem}
\begin{proof}If $T=\begin{pmatrix}a&^tx\\y&Y\end{pmatrix}$ with $a$
  in $\ZZ$, $x$ and $y$ in $\ZZ^n$ and $Y$ in $GL_{n}(\ZZ)$,  then
  equation (\ref{equ-T}) leads to the equalities
\begin{eqnarray}
\label{equ-ytv}y\cdot ^tv&=&0\\
\label{equ-atv=twX} a\, ^tv&=& ^tw\cdot Y.\end{eqnarray}
  From equation (\ref{equ-ytv})  we get that either $y$ or $v$
  vanishes. If $y=0$, then $T$ is triangular and $a\det Y=\pm 1$. In
  particular the matrix $Y$ is invertible, $a=\pm 1$ and from equation
 (\ref{equ-atv=twX}) we get that $v=a\, ^tY w$, so we can take $X=a\,
  ^t Y$.
If $v=0$, then since $T$ is invertible, $w=0$ and thus we can choose
$X=I_{n}$.
\end{proof}

\begin{thm}\label{thm-RKK}
 Let $X$ be a locally compact path connected space and let $A(X)$ and
  $B(X)$  be
  NCP $\TT^n$-bundles over $X$. Let $f,g:X\to
  \TT^\frac{n(n-1)}{2}$  be the continuous maps
  associated to $A(X)$ and $B(X)$, respectively, via the classification of NCP $\TT^n$-bundles 
of Theorem \ref{thm-T-bundle} and consider the following statements:
\begin{enumerate}
\item There exist principal $\TT^n$-bundles $Y$ and $Z$ over $X$ such
  that
$Y*A(X)$ and  $Z*B(X)$ are $\RKK$-equivalent.
\item The group bundles $\K_*(A)$ and $\K_*(B)$ are isomorphic as $\ZZ/2\ZZ$-graded group bundles.
\item $f$ is homotopic to $A\circ g$ for some $A\in GL(\Lambda_2(\ZZ^n))$.
\end{enumerate}
Then, if $n=2$, all these statements are equivalent. If $n=3$, then 
(i) $\Longrightarrow$  (ii) $\Longrightarrow$ (iii), and (iii) $\Longrightarrow$ (i) if $A\in SL(\Lambda_2(\ZZ^3))$. If $n>3$, then (i) $\Longrightarrow$ (ii) and (iii) $\Longrightarrow$ (i)
if $A=\Lambda_2\Psi$ for some $\Psi\in GL(n,\ZZ)$. 
\end{thm}

\begin{remark} \label{rem-RKK}
In case $n=3$ the obstruction for proving (iii) $\Longrightarrow$ (ii) is given by the problem, whether
the pullback $f^*C^*(H_3)$ is  $\RKK$-equivalent
to $C^*(H_3)$ as C*-algebra bundle over $\TT^3$, where $f:\TT^3\to \TT^3$ is the 
map $f(z_1,z_2,z_3)=(\bar{z}_1,\bar{z}_2,\bar{z}_3)$. We believe that this is true, but we 
are not able to prove it.
In  dimensions $n>3$ the situation is much more mysterious.
\end{remark}

\begin{proof} Since the assertions (i), (ii) and (iii) are left invariant under passing from 
$A(X)$ to $Y*A(X)$ (and similarly for $B(X)$), which is clear for (i) and (iii) and follows 
from  Lemma \ref{lem-group-twisted-bundle} for (ii), we may as well assume that $A(X)=f^*(C^*(H_n))$
and $B(X)=g^*(C^*(H_n))$.
Then (i) implies (ii) is a consequence of 
 Remark \ref{rem-KK}.

Let us now assume that condition (ii) is satisfied. Let us fix some $x\in X$. Since homotopies of NCP
torus bundles provide $\RKK$-equivalences by Corollary \ref{cor-homotopic},
and thus isomorphisms of the associated $K$-theory group bundles by Remark \ref{rem-KK}, we can
replace $f$ and $g$ by homotopic functions. Moreover,  $\TT^\frac{n(n-1)}{2}$ being a
connected group, we can in fact assume that $f(x)=g(x)=1$.
Since the group bundles $\K_*(f^*C^*(H_2)))$ and $\K_*(g^*C^*(H_2))$ are isomorphic, 
the corresponding actions of $\pi_1(X)$ on $K_0(C(\TT^\frac{n(n-1)}{2}))\cong\Lambda_{ev}(\ZZ^n)$
are  conjugate by a matrix $T
\in GL(\Lambda_{ev}(\ZZ^n))$ (note that the action of $\pi_1(X)$ on $K_*(C(\TT^n))\cong \Lambda_*(\ZZ^n)$ preserves the grading). 
If $n=2$, it follows then from  Proposition \ref{prop-pi1} 
 that for any $\gamma$ in $\pi_1(X)$, we
have
{$$T\cdot\begin{pmatrix}1&\langle f
  ,\gamma\rangle\\0&1\end{pmatrix}=\begin{pmatrix}1&\langle
  g,\gamma\rangle\\0&1\end{pmatrix}\cdot T,$$}
which implies by Lemma \ref{lem-conjugate} that
$\langle
  f,\gamma\rangle=\eps \langle
  g,\gamma\rangle$ for every $\gamma$ in $\pi_1(X)$, with $\eps=\pm 1$.
{Since  $\langle f \cdot g^{-\eps},\gamma\rangle =\langle
f,\gamma\rangle-\eps\langle g,\gamma\rangle$, this implies arguing as in the proof of
Proposition \ref{prop-inj} that $f$ and $g^\eps$ are
  homotopic and  hence (iii).}
 
 In case $n=3$ we can argue similarly: using Lemma \ref{lem-dim3} together with 
 Lemma \ref{lem-conjugate}, we see, as in the case $n=2$, that (ii) implies 
 the existence of some $A\in GL(\ZZ^3)$ ($=GL(\Lambda_2(\ZZ^3))$) such that
 $$\left(\begin{matrix} \langle g_1,\gamma\rangle \\ \langle g_2,\gamma\rangle \\ \langle g_3,\gamma\rangle \end{matrix}\right)
 =A\left(\begin{matrix} \langle f_1,\gamma\rangle \\ \langle f_2,\gamma\rangle \\ \langle f_3,\gamma\rangle \end{matrix}\right)
 =\left(\begin{matrix} \langle (A\circ f)_1,\gamma\rangle \\ \langle (A\circ f)_2,\gamma\rangle \\ \langle (A\circ f)_3,\gamma\rangle 
 \end{matrix}\right).$$
 Hence, combining Lemma \ref{lem-prod} with {the proof of } Proposition \ref{prop-inj}
this implies that $g$ is homotopic to $A\circ f$ for some $A$ in $GL(\Lambda_2(\ZZ^3))$.

 Let us finally assume (iii). As mentioned before we may
 assume that $A(X)=f^*(C^*(H_n))$ and $B(X)=g^*(C^*(H_n))$. If $g=A\circ f$ with
 $A=\Lambda_2\Psi$ for some $\Psi\in GL_n(\ZZ)$, then it follows from 
 Theorem \ref{thm-T-bundle-action} that 
 $g^*(C^*(H_n))=(\Lambda_2\Psi\circ f)^*(C^*(H_n))$ is isomorphic to $f^*(C^*(H_n))$
 as a C*-algebra bundles over $X$. In particular, it follows that 
 $f^*(C^*(H_n))$ is $\RKK$-equivalent to $g^*(C^*(H_n))$. 
 In case $n=2$, the range of $\Psi\mapsto \Lambda_2\Psi$ is all of 
 $GL(\Lambda_2(\ZZ^2))=GL_1(\ZZ)$. In case $n=3$, the range is $SL(\Lambda_2(\ZZ))$
 (see Lemma \ref{lem-image}). This finishes the proof.
 \end{proof}

\section{Applications to $T$-duality with H-flux}\label{sec-T-dual}
In this section we want to discuss an application of our results 
to the study of $T$-duals with H-flux, as studied by Mathai and Rosenberg in 
\cite{MR1,MR2}. We shall restrict ourselves to the mathematical point of view 
 and refer to the above cited papers  
(and the references given there) for physical interpretations of the theory.

Following the approach of \cite{MR1,MR2} we start with a fixed principal 
$\TT^n$-bundle $q:Y\to X$ together with a Dixmier-Douady class $\delta\in H^2(X,\ZZ)$
(which is called an {\em H-flux} in \cite{MR1}). 
It is then shown in \cite{MR2} that a (commutative or non-commutative) T-dual
only exists if the restriction of $\delta$ to the fibres $Y_x\cong \TT^n$ are trivial.
In that case the corresponding stable continuous trace algebra $CT(Y,\delta)$ carries 
a canonical structure as a C*-algebra bundle over $X$ with fibres 
$C(Y_x,\K)\cong C(\TT^n,\K)$
for all $x\in X$, and there exists at least one action {$\beta:\RR^n\to \Aut(CT(Y,\delta))$
which covers the given action on the base $Y$. A T-dual
in the sense  of Mathai-Rosenberg is then given by the crossed product 
$CT(Y,\delta)\rtimes_{\beta}\RR^n$. 

If the action $\beta$ can be chosen in such a way that the crossed product 
$CT(Y,\delta)\rtimes \RR^n$ is  again a continuous-trace algebra, then it is
shown in \cite{MR1,MR2}, using much earlier results from \cite{RR},
 that it is isomorphic to an algebra of the form $CT(\widehat{Y},\widehat{\delta})$
 for a $\TT^n$-principal bundle $\hat{q}:\widehat{Y}\to X$ and a suitable 
 Dixmier-Douady class $\widehat{\delta}\in H^3(\widehat{Y},\ZZ)$, such that 
 the dual action of $\RR^n$ on $CT(Y,\delta)\rtimes \RR^n\cong CT(\widehat{Y},\widehat{\delta})$
 covers the action on the base as explained before. In this case the bundle
 $CT(\widehat{Y},\widehat{\delta})$ is called a {\em classical T-dual} of $CT(Y,\delta)$.
 It is shown in \cite[Proposition 3.3]{MR2} that classical duals, if they exist, are essentially unique.

 The main result of \cite{MR2} is then a characterization of those bundles $CT(Y,\delta)$ which do 
 have a classical dual in the above sense (see \cite[Theorem 3.1]{MR2}). As an application of 
 our results we shall give a purely $K$-theoretical characterization of existence 
 of classical duals. Note first that all T-duals 
are at least locally isomorphic to NCP $\TT^n$-bundle 
in our sense. We start with}

\begin{lem}\label{lem-NCPdual}
Suppose that $Y=X\times \TT^n$ and $\delta|_{X\times\{1\}}$ is trivial. Suppose further that 
$\beta:\RR^n\to \Aut(CT(X\times \TT^n,\delta))$ is an action which covers the 
obvious action of $\RR^n$ on the base $X\times\TT^n$. Then 
there exists an action $\alpha:\ZZ^n\to \Aut(C_0(X,\K))$ such that
$$CT(X\times \TT^n,\delta)\rtimes_{\beta}\RR^n\sim_M C_0(X,\K)\rtimes_{\alpha}\ZZ^n$$
where here $\sim_M$ denotes Morita equivalence of C*-algebra bundles over $X$. In particular,
the T-dual $CT(X\times \TT^n,\delta)\rtimes_{\beta}\RR^n$ is Morita equivalent to an NCP $\TT^n$-bundle
over $X$.
\end{lem}
\begin{proof} If we restrict $\beta$ to $\ZZ^n$, we obtain a fibre-wise action of $\ZZ^n$ 
on $CT(X\times \TT^n,\delta)$, which then restricts to a fibre-wise action $\alpha$  of
$\ZZ^n$ on $C_0(X,\K)= CT(X\times\{1\},\delta|_{X\times\{1\}})$. It follows then from 
 \cite[Theorem]{ech} that 
  the systems 
$$\big(CT(X\times \TT^n,\delta),\,\RR^n,\,\beta\big)\quad\text{and}\quad
\big(\Ind_{\ZZ^n}^{\RR^n}C_0(X,\K),\,\RR^n\,,\Ind\alpha\big)$$
are isomorphic, and then  Green's imprimitivity theorem gives the desired 
Morita equivalence, which is easily checked to be compatible with the bundle
structures over $X$.
\end{proof}

As an easy consequence we  get:

\begin{prop}\label{prop-local-T-dual}
Assume that $q:Y\to X$ is a principal $\TT^n$-bundle with H-flux $\delta\in H^3(Y,\ZZ)$ 
and let $\beta:\RR^n\to \Aut(CT(Y,\delta))$ be an action which covers the given action on $Y$.
Then for any $y\in Y$
there exists an $\RR^n$-invariant neighbourhood $V$ of $y$
such that 
$CT(V,\delta|_V)\rtimes \RR^n$ is $C_0(U)$-Morita equivalent to some 
NCP $\TT^n$-bundle over $U$.
\end{prop}
\begin{proof} We first choose a trivializing  $\RR^n$-invariant neighbourhood $V$ of $y$ 
such that $V\cong U\times \TT^n$ as $\RR^n$-space. By passing to a smaller neighbourhood
of $x=q(y)\in U$ if necessary,  we may assume without loss of generality that $\delta|_U$ is trivial.
Hence the result follows from Lemma \ref{lem-NCPdual} above.
\end{proof}

\begin{cor}\label{cor-RKKtriv}
Let $q:Y\to X$ be a principal $\TT^n$-bundle with H-flux $\delta\in H^3(X,\ZZ)$ such that there 
exists an action $\beta$ of $\RR^n$ on $CT(Y,\delta)$ which covers the given action on $Y$.
Consider $CT(Y,\delta)$ and the T-dual $CT(Y,\delta)\rtimes \RR^n$ as C*-algebra bundles over $X$.
Then both bundles are are locally $\RKK$-trivial and 
$\RKK(X;\cdot, \cdot)$-equivalent to each other (up to a dimension shift of order $n$).
\end{cor}
\begin{proof} 
This follows from  Theorem \ref{thm-Thom}, Proposition \ref{prop-local-T-dual} and 
Corollary \ref{cor-localRKK-triv}.
\end{proof}

\begin{lem}\label{lem-triv}
Let $q:Y\to X$ be a principal $\TT^n$-bundle over $X$ and let $\delta\in H^3(Y,\ZZ)$ 
such that $\delta$ vanishes on the fibres $Y_x\cong \TT^n$. Assume further that
$\Delta$ is a compact contractible space and that $F:\Delta\to X$ is continuous map. Then
$$F^*CT(Y,\delta)\cong C(\Delta\times \TT^n,\K)$$
as C*-algebra bundles over $X$, such that the corresponding homeomorphism 
of the base $F^*Y\cong \Delta\times \TT^n$ is $\TT^n$-equivariant. In particular, 
$CT(Y,\delta)$ is a $K$-fibration in the sense of Definition \ref{def-K-fibration}.
\end{lem}
\begin{proof} The result follows easily from the fact that
 $F^*(CT(Y,\delta))$ is a continuous trace algebra over $F*Y=\Delta\times \TT^n$ (since $\Delta$ is contractible) such 
that the Dixmier-Douady class vanishes on the fibre $\TT^n$. Then a classical result 
on fibre bundles (e.g. see \cite{Huse}) 
implies that that the underlying  bundle of compact operators over $\Delta\times \TT^n$ 
corresponding to $F^*(CT(Y,\delta))$ is trivial, and hence
$F^*(CT(Y,\delta))\cong C(\Delta\times\TT^n,\K)$. 
\end{proof}

It follows from this lemma that whenever $q:Y\to\TT$ is a principal $\TT^n$-bundle with
H-flux $\delta$ such that the restriction of $\delta$ to the fibres is trivial,
 we can consider the $K$-theory group bundle 
$\K_*(CT(Y,\delta))$ over $X$ as introduced in Definition \ref{defn-group-bundle}.

If $\beta:\RR^n\to \Aut(CT(Y,\delta))$ is an action which covers the given $\RR^n$-action on $Y$,
then the restriction of $\beta$ to $\ZZ^n\subseteq \RR^n$ acts trivially on $Y$, and therefore
is a fibre-wise action on $CT(Y,\delta)$. Since the fibres are isomorphic to the 
compact operators on separable Hilbert space, we get, as explained in \S 2,  a 
Mackey-obstruction $[\om_y]\in H^2(\ZZ^n,\TT)$
which classifies the action on the fibre $\K_y$ of $CT(Y,\delta)$. The resulting map
$$\tilde{f}:Y\to H^2(\ZZ^n,\TT)=\TT^\frac{n(n-1)}{2}$$
is constant on $\RR^n$-orbits (see \cite[Theorem 2.3]{Goot}), 
and hence on the fibres $Y_x=q^{-1}({x})$ of the
principal bundle $q:Y\to X$. Thus $\tilde{f}$ factors through a well defined 
map 
$$f: X\to \TT^\frac{n(n-1)}{2}$$
which we call the Mackey-obstruction map for the T-dual $CT(Y,\delta)\rtimes_{\beta}\RR^n$.
Of course, if we locally realize the T-dual as an NCP $\TT^n$-bundle as 
in Proposition \ref{prop-local-T-dual}, then $f$ restricts to the Mackey-obstruction map 
of that NCP $\TT^n$-bundle as it shows up in the classification data of those bundles (see \S 2).
We are now ready to formulate the main result of this section, which builds on the 
work of Mathai and Rosenberg in \cite{MR2}:

\begin{thm}\label{thm-T-dual}
Suppose that $q:Y\to X$ is a principal $\TT^n$-bundle with H-flux $\delta\in H^3(Y,\ZZ)$ such that
$\delta$ vanishes on the fibres $Y_x\cong \TT^n$. Assume that $X$ is path connected. Then the following are equivalent:
\begin{enumerate}
\item There exists a classical T-dual for $CT(Y,\delta)$.
\item The action of $\pi_1(X)$ on the fibres $K_*(C_0(Y_x,\K))\cong K^*(\TT^n)$ is trivial.
\item The $K$-theory group bundle $\K_*(CT(Y,\delta))$ over $X$ with fibres 
$K_*(C_0(Y_x,\K))\cong K^*(\TT^n)$ is trivial.
\item For any action $\beta:\RR^n\to \Aut(CY(Y,\delta))$ which covers the given action on $Y$, 
the corresponding Mackey-obstruction map $f:X\to \TT^\frac{n(n-1)}{2}$ is homotopic to a constant.
\end{enumerate}
\end{thm}
\begin{proof} The equivalence of (ii) and (iii) is obvious and the equivalence of (i) and (iv) 
has been shown by Mathai and Rosenberg in \cite[Theorems 2.3 and 3.1]{MR2}. 
It thus remains to show the 
equivalence of (ii) and (iv). If $\beta$ is an action as in (iv) we have seen in 
Corollary \ref{cor-RKKtriv} that $CT(Y,\delta)$ is $\RKK$-equivalent to the T-dual
$CT(Y,\delta)\rtimes_{\beta}\RR^n$, and hence they have isomorphic 
K-theory bundles and conjugate actions of $\pi_1(X)$ on the fibres. 

Choose a fixed base-point $x$ and let $\nu:\TT\to X$ be a continuous map 
with $\nu(1)=x$ representing a given class $\gamma\in \pi_1(X,x)$.
Then $\nu^*Y$ is isomorphic to the trivial bundle $\TT\times \TT^n$, since 
$H^2(\TT,\ZZ)=\{0\}$. 
It follows that $\nu^*CT(Y,\delta)$ is isomorphic to $CT(\TT\times \TT^n, \nu^*\delta)$
as a C*-algebra bundle over $\TT$. Since $H^3(\TT,\ZZ)=\{0\}$, it is clear that 
$\nu^*\delta|_{\TT}$ vanishes, and thus it follows from Lemma \ref{lem-NCPdual}
that 
$$\nu^*\big(CT(Y,\delta)\rtimes_{\beta}\RR^n\big)=CT(\TT\times \TT^n, \nu^*\delta)\rtimes_{\nu^*\beta}\RR^n
\sim_M C(\TT,\K)\rtimes_{\alpha}\ZZ^n$$
where $\alpha$ is given by the restriction of $\nu^*\beta$ to $\ZZ^n$ evaluated at $\TT$.
In particular, the Mackey-obstruction map  $f_{\nu}=f\circ \nu$ of $\nu^*\beta$ coincides 
with the Mackey-obstruction map for $\alpha$. 
We then get 
$$\nu^*\big(CT(Y,\delta)\rtimes_{\beta}\RR^n\big)\sim_M 
C(\TT,\K)\rtimes_{\alpha}\ZZ^n
\sim_M
f_{\nu}^*C^*(H_n),
$$
where the last Morita equivalence follows from our classification of NCP torus bundles 
together with the fact that all commutative torus bundles over $\TT$ are trivial.
We may then conclude as in \S 6 that the map $M: \pi_1(X)\to 
GL(K_*(C(Y_x,\K))\cong GL(K_*(C^*(H_n)_{f(x)}))$ is given on $\gamma$
precisely by the formula given in equation (\ref{eq-action}) (notice that the identification 
$GL(K_*(C(Y_x,\K)))\cong GL(K_*(C^*(H_n)_{f(x)})))$ can be done independently from
the given choice of $\gamma$).
In particular, it follows then from Proposition \ref{prop-inj} that this action is trivial
if and only if $f$ is homotopic to the identity, which now finishes the proof.
\end{proof}

\section{Some final comments}
As mentioned in the introduction, the main object of this paper is to present an 
interesting class  of examples of C*-algebra bundles which are $K$- or $\KK$-fibrations
in the sense of Definition \ref{def-K-fibration}. The structure of being such fibration provides 
a new topological  invariant, namely the $K$-theory group bundle of the fibration, and 
the results of this paper show that this invariant can be quite useful for the study 
of the structure of such non-commutative fibrations. In particular, the $K$-theory group bundle
is an obstruction for $\RKK$-triviality.

In the forthcoming paper \cite{ENO}, we shall show that there are many more 
interesting examples of $K$- or $\KK$-fibrations, and we shall proof a complete version 
of the Leray-Serre spectral sequence  for all such fibrations, which will be a much stronger 
obstruction for $\RKK$-equivalence. In fact, we shall see in that paper that 
in case of NCP $\TT^n$-bundles the spectral sequence combines with the results 
of this paper  give a complete description which NCP $\TT^n$-bundles are 
$\RKK$-trivial (i.e., $\RKK$-equivalent to $C_0(X\times\TT^n)$).

\def\mathcs{{\normalshape\text{C}}^{\displaystyle *}}

\end{document}